\documentclass[10pt]{svjour3}                     
\smartqed  

\usepackage{graphicx}
\usepackage{amsmath}
\usepackage{amssymb}
\usepackage{mathtools}
\usepackage{hyperref}
\usepackage{paralist}
\usepackage{color}
\usepackage{makecell}
\usepackage{multirow}
\usepackage{multicol} 
\usepackage{booktabs} 
\usepackage{algorithm,hyperref}
\usepackage{algpseudocode}
\usepackage{float}    
\usepackage{booktabs} 
\usepackage{siunitx}  
\usepackage{wrapfig}

\newcommand{\Id}{\mathbb{I}}
\newcommand{\R}{\mathbb{R}}
\newcommand{\Hil}{\mathcal{H}}
\newcommand{\Rext}{\R\cup\{+\infty\}}

\newcommand{\set}[1]{\left\{#1\right\}}
\newcommand{\sets}[1]{\{#1\}}

\newcommand{\norms}[1]{\Vert#1\Vert}
\newcommand{\iprod}[1]{\left\langle #1\right\rangle}
\newcommand{\iprods}[1]{\langle #1\rangle}
\newcommand{\Eproof}{\hfill $\square$}
\newcommand{\prox}{\mathrm{prox}}

\newcommand{\zer}[1]{\mathrm{zer}(#1)}
\newcommand{\mbf}[1]{\mathbf{#1}}

\newcommand{\mcal}[1]{\mathcal{#1}}

\newcommand{\Hc}{\mathcal{H}}
\newcommand{\Lc}{\mathcal{L}}

\newcommand{\Tc}{\mathcal{T}}

\newcommand{\BigOs}[1]{\mathcal{O}\big(#1\big)}

\newcommand{\pbest}[1]{\textbf{\color{blue}#1}}

\newcommand{\beforesec}{\vspace{-4ex}}
\newcommand{\aftersec}{\vspace{-2ex}}
\newcommand{\beforesubsec}{\vspace{-4ex}}
\newcommand{\aftersubsec}{\vspace{-2ex}}

\newcommand{\cell}[1]{{\!\!}#1{\!\!}}  

\newcommand{\revise}[1]{{\color{black}#1}} 

\begin{document}

\title{Provable Parameter-Free Fixed-Point Algorithms with Linear Convergence Rates}
\titlerunning{Provable Parameter-Free Fixed-Point Algorithms with Linear Convergence Rates}
\authorrunning{Q. Tran-Dinh, N. A. Pham, and M. T. Ha}


\author{Quoc Tran-Dinh$^{*}$, Pham Ngoc Anh$^{\dagger}$, and Ha Manh Tien$^{\ddagger}$}

\institute{
Quoc Tran-Dinh$^{*}$, Pham Ngoc Anh$^{\dagger}$, and Ha Manh Tien$^{\ddagger}$ \at
{$^{*}$Department of Statistics and Operations Research,\\
The University of North Carolina at Chapel Hill, USA;} \\
{$^{\dagger}$Laboratory of Applied Mathematics and Computing,} \\
{Posts and Telecommunications Institute of Technology, Hanoi, Vietnam;} \\
{$^{\ddagger}$Faculty of Basic Sciences and Foreign Languages,} 
{Fire and Rescue Academy, Hanoi, Vietnam.}\\
{\small\textbf{Emails:}} \texttt{quoctd@email.unc.edu, anhpn@ptit.edu.vn, tienhm@daihocpccc.edu.vn}
}

\date{} 

\maketitle

\vspace{-9ex}
\begin{abstract}
\small
In this paper, we develop provable parameter-free fixed-point algorithms for contractive mappings, with an emphasis on automatically exploiting hidden contractivity without requiring prior knowledge of the contraction factor. 
Our first method is a parameter-free variant of the Halpern fixed-point iteration with no line-search, bisection, or prior estimate of the contraction factor, while retaining essentially the same per-iteration computational cost as classical fixed-point schemes. 
We establish explicit linear convergence rates for both the fixed-point residual and the distance to the unique fixed point. 
The second algorithm is also a parameter-free Halpern method, but it requires a line-search procedure to find a local estimate for the contraction factor, which may need a few additional operator evaluations. 
This method also enjoys explicit linear convergence guarantees. 
We further extend these ideas in two directions. 
First, using the relation between Halpern iterations and Nesterov-type accelerated fixed-point schemes, we derive parameter-free Nesterov-type (three-term) variants that inherit linear convergence in the contractive setting.
Second, by combining the proposed fixed-point schemes with Tikhonov regularization, we obtain new methods for solving cocoercive equations and establish an iteration complexity of $\mathcal{O}({\epsilon^{-1}\log(\epsilon^{-1})})$ for computing an $\epsilon$-solution, provided that a valid lower bound on the cocoercivity constant is available. 
Numerical experiments on several examples demonstrate that the proposed algorithms are competitive with, and often outperform, existing adaptive fixed-point methods. In particular, the methods successfully exploit contractive behavior when it is present on the selected nonexpansive test problems.
\end{abstract}

\vspace{-2ex}
\section{Introduction}\label{sec:intro}
\aftersec
Developing parameter-free algorithms for optimization and related problems such as minimax optimization and variational inequalities has attracted renewed interest in recent years, largely due to their strong empirical performance across a broad range of applications, particularly in machine learning and data science \cite{Barzilai1988,carmon2022making,defazio2022parameter,defazio2023learning,Duchi2011,fercoq2019adaptive,ito2023parameter,KingmaB14,lan2026optimal,li2025simple,malitsky2020adaptive,orabona2016coin}. 
In contrast, the development of parameter-free methods for fixed-point problems has received considerably less attention. 
Although several parameter-free fixed-point algorithms have recently been proposed, their practical performance has been investigated only to a limited extent, with relatively few numerical studies reported in the literature (see, e.g., \cite{alakoya2021modified,diakonikolas2020halpern,lv2026preconditioned,malitsky2020forward,ogwo2025inertial,shen2026parameter,tan2022self}). 
Moreover, many of these methods still lack rigorous theoretical guarantees. 
Most existing results establish only asymptotic convergence, without providing explicit convergence rates or iteration-complexity estimates. 
This noticeable gap between the optimization and fixed-point literature motivates the present work.

\vspace{0.75ex}
\noindent\textbf{$\mathrm{(a)}$~\textit{Our motivation and research question}.}
Besides the success of parameter-free methods in optimization, a second motivation stems from a practical observation. 
In many applications of fixed-point methods, the underlying mapping is simply assumed to be nonexpansive, and algorithms designed for nonexpansive mappings are subsequently employed to approximate its fixed point. 
In practice, however, it is often difficult, and sometimes impossible, to determine whether the mapping is actually contractive, even locally. 
If the mapping is indeed contractive, then one would naturally expect an algorithm to enjoy a linear convergence rate. 
In contrast, for general nonexpansive mappings, the optimal uniform worst-case rate for the fixed-point residual in the oracle setting considered in \cite{park2022exact} is of order $1/k$.
This naturally leads to the following question:
\begin{center}
\textit{\textbf{Can we develop efficient parameter-free algorithms that automatically exploit hidden contractivity without requiring prior knowledge of the contraction factor?}}
\end{center}
To explore this question, we run a preliminary experiment on two mappings in $\mathbb{R}^{20}$: (i) an affine mapping $T_1(x)=Q_1x+q_1$ and (ii) a nonlinear mapping
$T_2(x)=c\arctan(Q_2x+q_2)$ with $c=0.99$.
We generate $Q_1$ and $Q_2$ using scaled random-graph Laplacians so that both mappings are $\rho$-contractive with $\rho=0.9$, and choose $q_1$ and $q_2$ to prescribe their fixed points.
We choose this family because its structure is favorable to adaptive and parameter-free methods.

We compare our five variants with Algorithm~3.1 of \cite{he2024convergence}, called \texttt{He-etal24}, designed for nonexpansive mappings, and the geometric Halpern method of \cite{park2022exact}, called \texttt{Geo-HP},
which uses the known contraction factor $\rho$.
All methods start from an initial point $x^0=\mathbf{1}_p$ and use a relative residual criterion $\norms{x^k - T(x^k)} \leq 10^{-8}\max\sets{1, \norms{x^0 - T(x^0)}}$ to terminate.

\begin{figure}[!h]
    \centering
    \includegraphics[width=\textwidth]{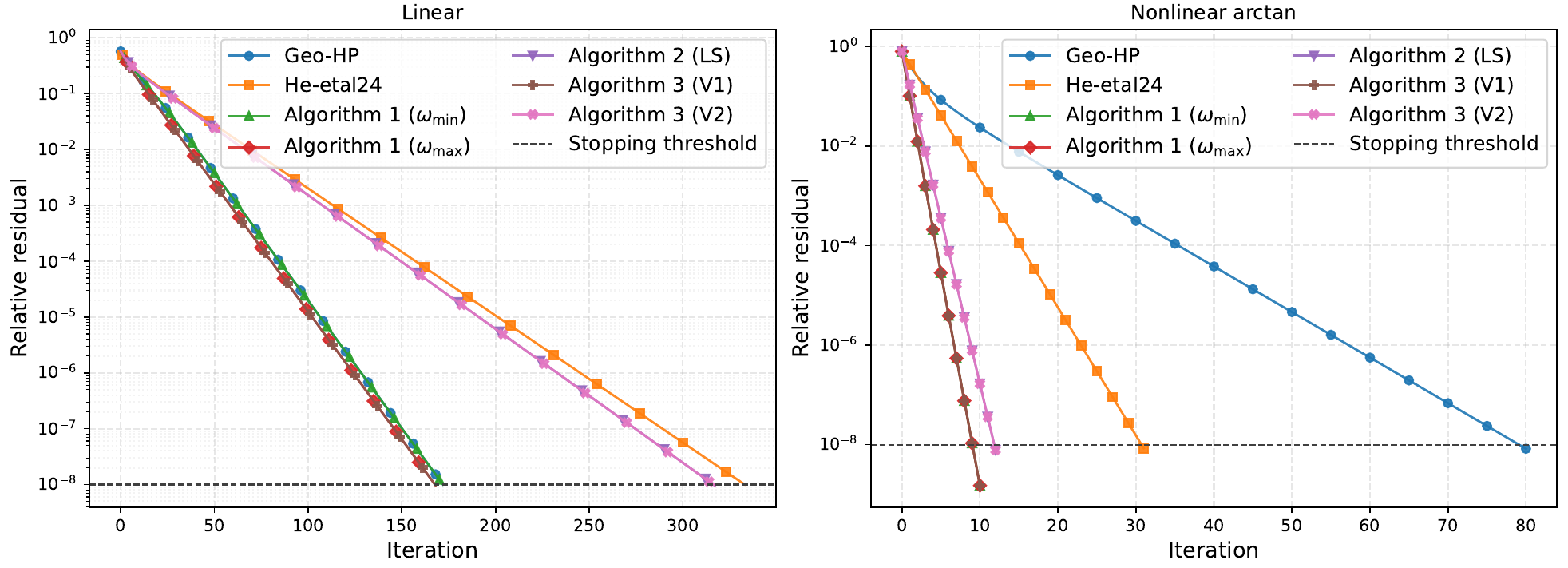}
    \vspace{-3ex}
    \caption{Relative fixed-point residual versus iteration count
    for the seven algorithms on an affine mapping (left) and a
    nonlinear arctangent mapping (right), with $p=20$ and $\rho=0.9$.}
    \label{fig:paired_convergence}
    \vspace{-3ex}
\end{figure}

Figure~\ref{fig:paired_convergence} shows that all seven methods reach the prescribed tolerance. 
The five proposed variants exhibit behavior consistent with linear convergence and require fewer iterations than Algorithm~3.1 of \cite{he2024convergence}
on both instances. 
Across $30$ additional random seeds, each proposed variant also uses fewer evaluations of $T$ than this baseline on both mappings.
Although the advantage is specific to this selected experiment, these observations illustrate how the proposed methods can exploit contractivity without knowing $\rho$ to achieve a fast linear convergence rate.

One may argue that, when $T$ is contractive, it suffices to apply the classical Banach--Picard (BP) iteration, which converges strongly and linearly to the unique fixed point. 
While this is true under global contractivity, verifying such a property is often difficult in practice. 
In many applications, it may be unclear whether the underlying mapping is globally contractive, only locally contractive, or merely nonexpansive. 
In particular, if contractivity holds only locally, the BP iteration may fail to converge when initialized outside the contractive region.
For a nonexpansive mapping with a nonempty fixed-point set, the Krasnosel'ski\v{\i}--Mann (KM) iteration with a fixed relaxation parameter in $(0,1)$ converges weakly to a fixed point, whereas the classical Halpern iteration with a standard admissible anchoring sequence converges strongly. 
These observations motivate us to develop Halpern-type methods that can automatically exploit contractivity when it is present.
However, we emphasize that our current theoretical guarantees are established only for contractive mappings. 
Nevertheless, our numerical experiments in Subsection~\ref{subsec:exam2} indicate that the proposed methods also perform well on merely nonexpansive mappings.

\vspace{0.75ex}
\noindent\textbf{$\mathrm{(b)}$~\textit{Problem statements}.}
Let $T : \Hil \to \Hil$ be a mapping defined on a real Hilbert space $\Hil$.
We are interested in approximating a fixed point of $T$, namely:
\begin{equation}\label{eq:FP}
\textrm{Find $x^{\star}\in \Hil$ such that:}~ x^{\star} = T(x^{\star}).
\tag{FP}
\end{equation}
Denote by $\mathrm{Fix}(T) := \sets{x^{\star} \in \Hil : x^{\star} = T(x^{\star}) }$ the set of fixed points of $T$.

Alternatively, we also consider the following root-finding problem:
\begin{equation}\label{eq:CE}
\textrm{Find $x^{\star}\in \Hil$ such that:}~ G(x^{\star}) = 0,
\tag{CE}
\end{equation}
where $G : \Hil \to \Hil$ is a single-valued mapping.
Denote by $\mathrm{zer}(G) := \sets{x^{\star} \in \Hil : G(x^{\star}) = 0}$ the solution set of \eqref{eq:CE}.
Throughout this paper, we assume that both $\mathrm{Fix}(T)$ and $\mathrm{zer}(G)$ are nonempty.

Problems \eqref{eq:FP} and \eqref{eq:CE} are closely related. Indeed, if $x^{\star} \in \mathrm{Fix}(T)$, then $x^{\star} \in \mathrm{zer}(\Id - T)$. 
Conversely, if $x^{\star} \in \mathrm{zer}(G)$, then $x^{\star} \in \mathrm{Fix}(\Id - \eta G)$ for any $\eta > 0$. 
Consequently, the fixed-point problem \eqref{eq:FP} can be viewed as an equivalent form of the root-finding problem \eqref{eq:CE}.

To develop numerical methods for solving \eqref{eq:FP}, it is common to assume that $T$ is either nonexpansive or $\rho$-contractive for some $\rho \in (0,1)$, see Section~\ref{sec:prelim_results} for definitions. 
Throughout this paper, we only focus on the contractive setting, where $T$ is $\rho$-contractive with $\rho \in (0,1)$. 
For the root-finding problem \eqref{eq:CE}, a standard assumption is that $G$ is $\beta$-cocoercive, i.e., $\iprods{G(x) - G(y), x - y} \geq \beta\norms{G(x) - G(y)}^2$ for all $x, y\in\Hil$.
We refer to \eqref{eq:CE} as a cocoercive equation.

The notions of nonexpansiveness and cocoercivity are closely connected. 
It is well-known that if $G$ is $\beta$-cocoercive, then $T := \Id - 2\beta G$ is nonexpansive. 
Conversely, if $T$ is nonexpansive, then $G := \Id - T$ is $\frac{1}{2}$-cocoercive. 
Therefore, solving a cocoercive equation is equivalent to approximating a fixed point of a nonexpansive mapping (or more generally, an averaged mapping), see \cite{Bauschke2011}.

For the fixed-point problem~\eqref{eq:FP}, if $T$ is globally contractive, i.e., $\norms{T(x) - T(y)} \leq \rho\norms{x - y}$ for all $x, y \in \Hil$ with $\rho \in (0, 1)$, then the classical Banach--Picard (BP) iteration \cite{Banach1922} converges linearly and strongly to the unique fixed point. 
When $T$ is only nonexpansive, two classical fixed-point methods are commonly employed. 
The first is the Krasnosel'ski\v{\i}--Mann (KM) iteration \cite{Krasnoselskii1955,mann1953mean}, which achieves the sublinear convergence rate $\BigOs{1/\sqrt{k}}$ in terms of the residual norm $\norms{x^k - T(x^k)}$. 
The second is the Halpern fixed-point iteration \cite{halpern1967fixed}, which often achieves a better convergence rate $\BigOs{1/k}$ under an appropriate choice of the anchoring parameter.

For convenience, we first briefly recall the Halpern fixed-point iteration from \cite{halpern1967fixed}.
Starting from an initial point $x^0 \in \Hil$ and an anchor point $u\in \Hil$ (which may coincide with $x^0$), at each iteration $k \geq 0$, the Halpern fixed-point iteration generates a sequence of iterates $\sets{x^k}_{k\geq 0}$ as follows: 
\begin{equation}\label{eq:HP_iteration_rule}
x^{k+1} = \lambda_k u + (1-\lambda_k)T(x^k),
\end{equation}
where $\lambda_k \in (0, 1)$ is an anchoring parameter. 
Various choices of $\lambda_k$ have been proposed to guarantee the convergence of \eqref{eq:HP_iteration_rule}, see, e.g., \cite{he2024convergence}. 
More recently, \cite{lieder2021convergence,sabach2017first} showed that the choice $\lambda_k = \frac{1}{k+2}$ yields the non-asymptotic convergence rate $\BigOs{1/k}$ for the residual norm $\norms{x^k - T(x^k)}$. 
Furthermore, this rate is known to be optimal, see, for example, \cite{lieder2021convergence,park2022exact}.

The development of parameter-free and adaptive Halpern methods is relatively recent. 
Diakonikolas \cite{diakonikolas2020halpern} appears to be among the first to propose a parameter-free variant of \eqref{eq:HP_iteration_rule} for solving the cocoercive equation \eqref{eq:CE}. 
The proposed method relies on a line-search procedure usually used in optimization and variational inequality problems (VIPs) to seek an appropriate bound on the cocoercivity parameter. 
More recently, He et al.~\cite{he2024convergence} proposed the following simple adaptive update rule for $\lambda_k$ in \eqref{eq:HP_iteration_rule}, which avoids both line-search and bisection:
\begin{equation}\label{eq:adaptive_stepsize}
\lambda_k = \frac{1}{1+\varphi_k}, \quad\text{where} \quad \varphi_k = 1 + \frac{2\iprods{x^k - T(x^k), x^0 - x^k}}{\norms{x^k - T(x^k)}^2}.
\end{equation}
\revise{Under this update rule, the authors established an $\BigOs{1/k}$ bound for the fixed-point residual $\norms{x^k - T(x^k)}$ and proved strong convergence of the iterate sequence $\sets{x^k}$ to a fixed point of $T$.}
However, this rule is designed for nonexpansive mappings: its residual rate remains $\BigOs{1/k}$ as in the standard Halpern fixed-point iteration studied in  \cite{lieder2021convergence,park2022exact} and it provides no contractivity-dependent guarantee.

Motivated by and complementing the work of \cite{diakonikolas2020halpern,he2024convergence}, we develop in this paper two new algorithms for approximating a fixed point of \eqref{eq:FP} when $T$ is assumed to be $\rho$-contractive for a some $\rho \in (0, 1)$, though we do not need to know the exact contraction factor $\rho$. 
The first algorithm is inspired by the geometric Halpern method of \cite{park2022exact}, while the second builds upon the adaptive parameter update introduced in \cite{he2024convergence}. 
Both methods are designed to exploit the contractivity of $T$ and are shown to achieve explicit linear convergence rates.

\vspace{0.75ex}
\noindent\textbf{$\mathrm{(c)}$~\textit{Our contributions}.}
The main contributions of this paper are  as follows.

\begin{compactitem}
\item[$\mathrm{(i)}$]
We propose a new parameter-free Halpern fixed-point algorithm for approximating the fixed point of \eqref{eq:FP}, where $T$ is a $\rho$-contractive mapping with $\rho \in (0,1)$. 
The proposed method does not require any prior knowledge of the contraction factor $\rho$. 
We establish explicit linear convergence rates both for the fixed-point residual $\norms{x^k - T(x^k)}$ and the distance to the fixed point $\norms{x^k - x^{\star}}$.

\item[$\mathrm{(ii)}$]
Inspired by \cite{he2024convergence}, we develop a new adaptive Halpern fixed-point algorithm for approximating the fixed point of \eqref{eq:FP}, where $T$ is a $\rho$-contractive mapping with $\rho \in (0,1)$. 
This algorithm remains parameter-free, but requires a line-search procedure to find a local estimate for $\rho$ to provably guarantee linear convergence.
We also establish explicit linear convergence rates for both the fixed-point residual $\norms{x^k - T(x^k)}$ and the distance to the fixed point $\norms{x^k - x^{\star}}$.

\item[$\mathrm{(iii)}$]
Exploiting the equivalence between Halpern's fixed-point iteration and Nesterov-type schemes established in \cite{tran2022connection}, we derive a new parameter-free Nesterov-type fixed-point algorithm for solving \eqref{eq:FP}. 
As a consequence of the results established in (i) and (ii), we prove that the proposed algorithm also enjoys a linear convergence rate whenever $\rho\in(0,1)$.

\item[$\mathrm{(iv)}$]
By combining the proposed fixed-point algorithms with Tikhonov regularization, we develop a regularized method, with a parameter-free inner iteration and a certified stepsize based on a valid cocoercivity lower bound, for solving the cocoercive equation \eqref{eq:CE} when $G$ is $\beta$-cocoercive. 
Building upon the convergence results established in (i) and (ii), we prove the convergence of the proposed method and derive the iteration complexity $\BigOs{\epsilon^{-1}\log(\epsilon^{-1})}$ for computing an $\epsilon$-solution $x^k$ satisfying $\norms{G(x^k)}\leq\epsilon$ for a given accuracy $\epsilon > 0$.
\end{compactitem}
\textbf{\textit{Discussion.}}
Let us further elaborate on the main contributions of this paper. 
First, our work complements the recent developments in \cite{diakonikolas2020halpern,he2024convergence} by establishing explicit linear convergence rates for contractive mappings. 
This distinguishes our results from many existing works, where only asymptotic convergence is available without explicit convergence rates or iteration-complexity guarantees.
Second, the first algorithm incurs essentially the same per-iteration cost as the classical Halpern, BP and KM iterations. 
The additional computation consists of a few vector operations, inner products, and norm evaluations, all of which require $\BigOs{p}$ operations when $\Hc = \R^p$.
The line-search variants may additionally need a few extra operator evaluations per iteration.
Third, our theoretical guarantees concern contractive mappings. 
However, Section~\ref{sec:num_examples} reports nonexpansive examples on which the proposed methods perform well (Examples 2.2--2.3) and one (Example 2.1) on which the line-search variants, as well as \texttt{Geo-HP} and \texttt{He-etal24}, do not reach the tolerance within the iteration budget. 
Fourth, by exploiting the equivalence established in \cite{tran2022connection}, our parameter rules can be transferred directly to Nesterov-type three-term (KM-with-correction) schemes.
This connection also suggests a promising direction for developing parameter-free variants of Nesterov-type fixed-point methods studied in the literature.
Fifth, by combining the proposed methods with Tikhonov regularization, we obtain new adaptive algorithms for solving cocoercive equations, or equivalently, fixed-point problems involving nonexpansive mappings. 
This substantially broadens the applicability of our approach to a variety of algorithms, including proximal-point, projection, and operator-splitting methods.
Finally, although the numerical experiments reported in this paper are based on synthetic examples, they provide encouraging evidence that the proposed methods have the potential to be effective for a broad range of practical applications.

\vspace{0.75ex}
\noindent\textbf{$\mathrm{(d)}$~\textit{Related work}.}
As mentioned earlier, parameter-free and adaptive algorithms have been extensively studied in optimization and have recently gained significant attention due to their success in machine learning and artificial intelligence, see, e.g., \cite{defazio2022parameter,defazio2023learning,Duchi2011,KingmaB14,orabona2016coin}. 
In contrast, despite several recent advances, the development of provably convergent and practically efficient parameter-free algorithms for fixed-point problems and problems in operator theory remains relatively limited.

Over the past decades, a number of researchers have investigated adaptive and parameter-free methods for fixed-point problems, as well as root-finding methods for monotone inclusions (MIs) and  VIPs. 
In fixed-point theory, several adaptive iterative schemes have been proposed. 
For example, Colao and Marino \cite{colao2015krasnoselskii} studied an adaptive variant of the Krasnosel'ski\v{\i}-Mann (KM) iteration by introducing adaptive update rules for the relaxation parameters. 
Diakonikolas \cite{diakonikolas2020halpern} proposed a parameter-free Halpern fixed-point method for solving the cocoercive equation \eqref{eq:CE}, achieving near-optimal (up to a logarithmic factor) complexity. 
More recently, He et al.~\cite{he2024convergence} developed an adaptive Halpern fixed-point method for solving \eqref{eq:FP} with nonexpansive mappings. 
Their method eliminates the need for line search while demonstrating promising numerical performance on several representative examples. 
Building upon this idea, Lv et al.~\cite{lv2026preconditioned} proposed a preconditioned Halpern method with adaptive parameter updates.

For monotone inclusions and VIPs, numerous adaptive algorithms based on line-search techniques have been developed; comprehensive accounts can be found in the monographs \cite{Facchinei2003,Konnov2001} and some recent works such as \cite{he2018totally,oyewole2022totally}. 
In particular, Malitsky and his collaborators proposed several adaptive extragradient-type and forward-backward splitting-type methods with and without line-search procedures for solving these problems, see, for example, \cite{malitsky2019golden,malitsky2020forward}. 
These algorithms have subsequently found applications in minimax optimization, game theory, Nash's equilibria, and machine learning, see, e.g., \cite{ioan2023relaxed,daskalakis2018training,malitsky2016first}.

\vspace{0.75ex}
\noindent\textbf{$\mathrm{(e)}$~\textit{Paper organization}.}
The remainder of this paper is organized as follows.
Section~\ref{sec:prelim_results} reviews the background material and presents some preliminary results used throughout the paper.
Section~\ref{sec:linear_convergence_rate} introduces our first parameter-free Halpern fixed-point algorithm and establishes its linear convergence.
Section~\ref{sec:adaptive_Halpern_method} develops an adaptive Halpern fixed-point algorithm together with its convergence analysis.
Section~\ref{sec:NesMethods} exploits the connection established in \cite{tran2022connection} to derive parameter-free Nesterov-type  variants of our methods.
Section~\ref{sec:NonExMapping} applies the proposed methods to solve cocoercive equations by combining them with Tikhonov regularization.
Finally, Section~\ref{sec:num_examples} presents numerical experiments demonstrating the performance of our algorithms and compares them with existing methods.

\beforesec
\section{Background and Preliminary Results}\label{sec:prelim_results}
\aftersec
We begin by reviewing the basic notation and concepts used throughout the paper, including fixed-point mappings and cocoercive equations. 
For further details, we refer to \cite{Bauschke2011,Facchinei2003}.
We then briefly recall Halpern's fixed-point iteration for approximating a fixed point of a $\rho$-contractive mapping.

\beforesubsec
\subsection{\textbf{Notation and basic concepts}}\label{sec:notation}
\aftersubsec
Throughout this paper, $\Hil$ denotes a real Hilbert space equipped with an inner product $\iprod{\cdot,\cdot}$ and the induced norm $\norms{\cdot}$.
For a function $f : \Hil \to\Rext$, $\nabla{f}$ denotes its gradient, $\partial{f}$ its subdifferential, and $\prox_{f}$ its proximal operator, when they are well-defined.
For a matrix $\mbf{Q}\in\R^{n\times p}$, we denote by $\norms{\mbf{Q}}_2$ its spectral norm.
\revise{Given two functions $g(t), h(t) : \R_{+} \to \R$, we write $g(t) = \BigOs{h(t)}$ as $t \to +\infty$ if there exist constants $M > 0$ and $t_0 \geq 0$ such that $\vert g(t)\vert \leq M \vert h(t) \vert$ for all $t \geq t_0$.}
We use $\widetilde{\mcal{O}}(g(t))$ to suppress polylogarithmic factors of $g(t)$.

Let $T : \Hil \to \Hil$ be a single-valued mapping.
We say that $T$ is $\rho$-Lipschitz continuous with Lipschitz constant $\rho > 0$ if $\norms{T(x) - T(y)} \leq \rho\norms{x - y}$ for all $x, y \in \Hil$.
If $\rho \in (0, 1)$, then $T$ is called $\rho$-contractive.
If $\rho = 1$, then $T$ is called nonexpansive.
Furthermore, if $\norms{T(x) - T(y)}^2 \leq \norms{x - y}^2 - \norms{x - T(x) - (y - T(y))}^2$ for all $x, y \in \Hil$, then $T$ is called firmly nonexpansive.
It is well-known that if $T$ is $\rho$-contractive with $\rho \in (0, 1)$, then $\textrm{Fix}(T)$ is nonempty and a singleton.

A single-valued mapping $G : \Hil \to \Hil$ is called $\beta$-cocoercive with cocoercivity parameter $\beta > 0$ if $\iprods{G(x) - G(y), x - y} \geq \beta\norms{G(x) - G(y)}^2$ for all $x, y \in \Hil$.
If $\iprods{G(x) - G(y), x - y} \geq 0$ for all $x, y \in \Hil$, then $G$ is called monotone.
It is well known that if $G$ is $\beta$-cocoercive, then $T := \Id - 2\beta G$ is nonexpansive.
Conversely, if $T$ is nonexpansive, then $G := \Id - T$ is $\frac{1}{2}$-cocoercive.
Moreover, every $\beta$-cocoercive mapping is also $\frac{1}{\beta}$-Lipschitz continuous.

\beforesubsec
\subsection{\textbf{Halpern fixed-point iteration and tight convergence rates}}\label{sec:HP_convergence}
\aftersubsec
Let $T : \Hil \to \Hil$ be a $\rho$-Lipschitz continuous mapping for some $\rho \in (0,1]$.
The Halpern fixed-point iteration for approximating a fixed point $x^{\star} \in \mathrm{Fix}(T)$ of $T$, i.e., $x^{\star} = T(x^{\star})$, is defined as follows.
Choose an anchor point $u \in \Hil$ and an initial point $x^0 \in \Hil$ (where $x^0$ may differ from $u$).
Then, for each iteration $k \geq 0$, update
\begin{equation}\label{eq:HP_iteration}
x^{k+1} := \lambda_k u + (1 - \lambda_k)T(x^k),
\tag{HP}
\end{equation}
where, following \cite{park2022exact}, the anchoring parameter $\lambda_k$ is chosen as
$\lambda_k = \frac{1}{1 + \varphi_k}$ with $\varphi_k := \sum_{i=1}^{k+1}\rho^{-2i}$.
For simplicity, throughout the remainder of the paper we set the anchor point to be the initial point, i.e., $u = x^0$.

The following theorem summarizes the convergence of the iteration \eqref{eq:HP_iteration}; see \cite[Corollary~3.3]{park2022exact}.

\begin{theorem}[\cite{park2022exact}]\label{th:linear_convergence_of_HP}
Let $T : \Hil \to \Hil$ be a $\rho$-contractive mapping with $\rho \in (0, 1)$ or a nonexpansive mapping with $\rho = 1$, and let $x^{\star} \in \mathrm{Fix}(T)$ be a fixed point of $T$.
Let $\sets{x^k}$ be the sequence generated by \eqref{eq:HP_iteration} with $x^0 := u$.
Then
\begin{equation}\label{eq:th11_convergence_bound}
\arraycolsep=0.2em
\begin{array}{lcl}
\norms{x^k - T(x^k)} & \leq & \frac{(1+\rho)}{\sum_{i=0}^{k}\rho^{-i} }\norms{x^0 - x^{\star}}  \leq  (1+\rho)\rho^k \norms{x^0 - x^{\star}}, \quad \forall k \geq 1.
\end{array}
\end{equation}
\end{theorem}

When $\rho \in (0,1)$, i.e., $T$ is contractive, Theorem~\ref{th:linear_convergence_of_HP} establishes a linear convergence rate for the fixed-point residual $\norms{x^k - T(x^k)}$. On the other hand, when $\rho = 1$, i.e., $T$ is nonexpansive, the convergence bound \eqref{eq:th11_convergence_bound} reduces to
\begin{equation*}
\norms{x^k - T(x^k)} \leq \frac{2\norms{x^0 - x^{\star}}}{k+1},
\end{equation*}
which yields the sublinear convergence rate $\BigOs{1/k}$. 
As shown in \cite{park2022exact}, the rate stated in Theorem~\ref{th:linear_convergence_of_HP} is optimal and cannot be improved in the worst case.
If $T$ is $\rho$-contractive for $\rho \in (0, 1)$, then, by the relation $\norms{x^k - x^{\star}} \leq \frac{1}{1-\rho}\norms{x^k - T(x^k)}$ (whose proof can be found in the proof of Theorem~\ref{th:convergence_theorem1}), we also get $\norms{x^k - x^{\star}} \leq \frac{(1+\rho)\rho^k }{1-\rho}\norms{x^0 - x^{\star}}$.
Therefore, $\sets{x^k}$ also strongly and linearly converges to the unique fixed point $x^{\star}$ of $T$.

\vspace{1ex}
\beforesec
\section{Main Result 1: Parameter-Free Halpern Fixed-Point Method}\label{sec:linear_convergence_rate}
\aftersec
In this section, motivated by Theorem~\ref{th:linear_convergence_of_HP}, we develop a parameter-free Halpern fixed-point method for approximating a fixed point of a $\rho$-contractive mapping $T$ for $\rho \in (0, 1)$. 
We first derive the proposed method and present its implementation in Algorithm~\ref{alg:adaptive_HP0}. We then establish its linear convergence.

\beforesubsec
\subsection{\textbf{The derivation of the algorithm}}\label{subsec:derivation_of_A1}
\aftersubsec
Our main goal is to develop a parameter-free variant of \eqref{eq:HP_iteration} for the contractive case.
Motivated by the update rule $\varphi_{k} := \sum_{i=1}^{k+1}\rho^{-2i} = \sum_{i=0}^k\rho^{-2(i+1)} = \varphi_{k-1} + \rho^{-2(k+1)}$  of $\varphi_k$ in \eqref{eq:HP_iteration} and the $\rho$-contractivity of $T$, it is natural to replace $\varphi_k$ by the following upper bound (when it is well-defined):
\begin{equation}\label{eq:new_stepsize0}
\bar{\varphi}_k := \bar{\varphi}_{k-1} + \frac{\norms{x^k - x^{k-1}}^{2(k+1)}}{\norms{T(x^k) - T(x^{k-1})}^{2(k+1)}},
\end{equation}
where $\bar{\varphi}_0 := \frac{1}{\rho_0^2}$ for a given initial value $\rho_0 \in (0, 1)$ as an estimate of $\rho$.
Indeed, by the $\rho$-contractivity  of $T$, we have $\norms{T(x^k) - T(x^{k-1})} \leq \rho\norms{x^k - x^{k-1}}$ for all $k \geq 1$.
If $\norms{T(x^k) - T(x^{k-1})} > 0$ for all $k \geq 1$, then $\frac{\norms{x^k - x^{k-1}}^{2(k+1)}}{\norms{T(x^k) - T(x^{k-1})}^{2(k+1)}} \geq \frac{1}{\rho^{2(k+1)}}$ for $k \geq 1$.
Therefore, if $\rho_0 \leq \rho$, then, by induction, we can claim that
\begin{equation}\label{eq:new_stepsize0_bound}
\bar{\varphi}_k \geq \sum_{i=0}^{k} \rho^{-2(i+1)} =: \varphi_k.
\end{equation}
However, we may encounter $\norms{T(x^k) - T(x^{k-1})} = 0$ for some $k \geq 1$.
To avoid this situation, we first notice that $\varphi_k = \varphi_{k-1} + \rho^{-2(k+1)}$ for all $k \geq 0$ and, by convention, $\varphi_{-1} = 0$.
Next, we can replace $\varphi_k$ by $\bar{\varphi}_k$ updated as follows:
\begin{equation}\label{eq:new_stepsize}
\bar{\varphi}_k := \bar{\varphi}_{k-1} + \omega_k^{2(k+1)}, \quad \textrm{where} \quad \omega_k \geq \frac{1}{\rho} \quad \textrm{for all $k \geq 0$ and $\bar{\varphi}_{-1} := 0$}.
\end{equation}
Clearly, $\bar{\varphi}_k$ is also an upper bound on $\varphi_k$.

Now, we first choose $\omega_0$ such that $\omega_0 \geq \frac{1}{\rho}$.
For this purpose, given $x^0 \in \Hil$, we find a direction $v \neq 0$ in $\Hil$ such that $T(x^0 + v) \neq T(x^0)$.
In our experiments, we choose $v := \frac{1}{8}\mathbf{1}_p$.
Then, by the $\rho$-contractivity of $T$, one way to choose the initial value $\omega_0$ is as follows:
\begin{equation}\label{eq:omega_0_choice}
\omega_0 := \frac{1}{\rho_0} \geq \frac{1}{\rho},\quad \textrm{where} \quad \rho_0 := \frac{\norms{T(x^0) - T(x^0 + v)}}{\norms{v}} \in (0,  \rho].
\end{equation}
For $k \geq 0$, given $\omega_k$, we can update $\omega_k$ as follows:
\begin{equation}\label{eq:update_omega_k}
\omega_{k+1} := \begin{cases}
\max\set{ \omega_k, \frac{\norms{x^{k+1} - x^k } }{ \norms{T(x^{k+1}) - T(x^{k})} } } & \textrm{if $\frac{\norms{x^{k+1} - x^k } }{ \norms{T(x^{k+1}) - T(x^{k})} } \leq \bar{\omega}$}, \\
\omega_k & \textrm{otherwise}.
\end{cases}
\end{equation}
If $\norms{T(x^{k+1}) - T(x^{k})} = 0$, then we keep $\omega_{k+1} = \omega_k$.
Here, we set an upper bound $\bar{\omega} \in (0, +\infty)$ on $\frac{ \norms{x^{k+1} - x^k } }{ \norms{T(x^{k+1}) - T(x^{k})} }$ to prevent the explosion of this ratio.
We can heuristically choose $\bar{\omega} := \frac{\Lambda}{\rho_0} > \frac{1}{\rho}$ for some $\Lambda > 1$ sufficiently large, e.g., $\bar{\omega} := \frac{2^4}{\rho_0}$.

Since $\omega_0 \geq \frac{1}{\rho}$ and $\frac{\norms{x^{k+1} - x^k } }{ \norms{T(x^{k+1}) - T(x^{k})} } \geq \frac{1}{\rho}$ whenever $\norms{T(x^{k+1}) - T(x^{k})} > 0$, it follows that $\omega_k \geq \frac{1}{\rho}$ for all $k \geq 0$.
Using $\bar{\varphi}_k$ from \eqref{eq:new_stepsize} to replace $\varphi_k$ in the scheme \eqref{eq:HP_iteration}, we obtain a parameter-free variant of this algorithm.

Note that we can also replace the update rule \eqref{eq:update_omega_k} by the following one:
\begin{equation}\label{eq:update_omega_k_modified}
\omega_{k+1} := \begin{cases}
\min\set{ \omega_k, \frac{\norms{x^{k+1} - x^k } }{ \norms{T(x^{k+1}) - T(x^{k})} } } & \textrm{if $\frac{\norms{x^{k+1} - x^k } }{ \norms{T(x^{k+1}) - T(x^{k})} } \leq \bar{\omega}$}, \\
\omega_k & \textrm{otherwise}.
\end{cases}
\end{equation}
Here, again, if $\norms{T(x^{k+1}) - T(x^{k})} = 0$, then we keep $\omega_{k+1} = \omega_k$.
Then the following properties hold: $\omega_k \geq \frac{1}{\rho}$ and $\omega_{k+1} \leq \omega_k$ for all $k \geq 0$.
The relation $\omega_{k+1} \leq \omega_k$ for all $k \geq 0$ is a direct consequence of \eqref{eq:update_omega_k_modified}.
We prove $\omega_k \geq \frac{1}{\rho}$ by induction. 
For $k = 0$, we have $\omega_0 := \frac{1}{\rho_0} \geq \frac{1}{\rho}$.
Suppose that  $\omega_k \geq \frac{1}{\rho}$ for some $k \geq 0$.
We show that $\omega_{k+1} \geq \frac{1}{\rho}$.
Indeed, from \eqref{eq:update_omega_k_modified} we have two possibilities: (i) $\omega_{k+1} = \omega_k \geq \frac{1}{\rho}$ or (ii) $\omega_{k+1} = \frac{\norms{x^{k+1} - x^k } }{ \norms{T(x^{k+1}) - T(x^{k})} } \geq \frac{1}{\rho}$.
Therefore, we obtain that $\omega_k \geq \frac{1}{\rho}$ for all $k \geq 0$. 

\beforesubsec
\subsection{\textbf{The full parameter-free algorithm}}\label{subsec:parameter_free_Algorithm1}
\aftersubsec
Now, we combine all the steps discussed in Subsection~\ref{subsec:derivation_of_A1} algorithmically to obtain a \textit{parameter-free Halpern fixed-point algorithm}, denoted by Algorithm~\ref{alg:adaptive_HP0}.

\vspace{-2ex}
\begin{algorithm}[!htbp]
\caption{(Parameter-Free Halpern Fixed-Point Algorithm)}\label{alg:adaptive_HP0}
\begin{algorithmic}[1]
\itemsep=0.2em
\State\textbf{Initialization:} 
Choose $x^0\in\mathcal H$.
Choose $v$ and $\bar{\omega}$, and compute $\omega_0$ by \eqref{eq:omega_0_choice}.
Set $\bar{\varphi}_{-1} := 0$.
\State \textbf{For $k = 0, 1, 2, \cdots$, perform}
\State\hspace{3ex}\label{alg1:step3}If ~$\norms{r^k} = 0$ for $r^k := x^k - T(x^k)$,  then TERMINATE.
\State\hspace{3ex}\label{alg1:step4}Update $\bar{\varphi}_{k} := \bar{\varphi}_{k-1} + \omega_k^{2(k+1)}$.
\State\hspace{3ex}\label{alg1:step5}Update
\begin{equation*}
x^{k+1} :=  \frac{1}{ \bar{\varphi}_k + 1} x^0 + \frac{ \bar{\varphi}_k}{\bar{\varphi}_k + 1}T(x^k).
\end{equation*}
\State\hspace{3ex}\label{alg1:step6}Update $\omega_{k+1}$ as in \eqref{eq:update_omega_k} or \eqref{eq:update_omega_k_modified}.
\State\textbf{End For}
\end{algorithmic}
\end{algorithm}
\vspace{-2ex}

\noindent
We make the following remarks regarding Algorithm~\ref{alg:adaptive_HP0}.
\begin{compactitem}[$\diamond$]
\itemsep=0.2em
\item Algorithm~\ref{alg:adaptive_HP0} is parameter-free, as it does not require any prior knowledge of the contraction factor $\rho$ of $T$. 
Moreover, it is straightforward to implement since it does not involve any inner loop, line-search, or bisection procedure.

\item In floating-point arithmetic, due to the power term $\omega_k^{2(k+1)}$ at Step~\ref{alg1:step4}, we compute $\lambda_k = \frac{1}{1 + \bar{\varphi}_k}$  in log-scale (or via ratios $\bar{\varphi}_{k-1}/\bar{\varphi}_k$) to avoid overflow.

\item The stopping criterion in Step~\ref{alg1:step3} is primarily intended for theoretical analysis. 
In practice, given a prescribed accuracy $\epsilon > 0$, we terminate Algorithm~\ref{alg:adaptive_HP0} whenever $\norms{r^k} \leq \epsilon$ or $\norms{r^k} \leq \epsilon\cdot\max\set{1, \norms{r^0}}$.
As a safeguard, we can also set the maximum number of iterations to $k_{\max}$.

\item The update rule for $\omega_{k+1}$ in Step~\ref{alg1:step6} requires only the evaluation of the two norms $\norms{x^{k+1} - x^k}$ and $\norms{T(x^{k+1}) - T(x^k)}$, together with two elementary arithmetic operations. 
Consequently, it introduces only negligible computational overhead compared with the classical Halpern iteration.
\end{compactitem}
Overall, the per-iteration computational complexity of Algorithm~\ref{alg:adaptive_HP0} is essentially the same as that of the classical Halpern, Banach--Picard (BP), and Krasnosel'ski\v{\i}--Mann (KM) fixed-point iterations.

\beforesubsec
\subsection{\textbf{Convergence analysis}}\label{subsec:convergence_guarantee_of_A1}
\aftersubsec
The following theorem proves linear convergence rates of Algorithm~\ref{alg:adaptive_HP0}.
The proof of this theorem is essentially elementary and does not require the construction of a sophisticated Lyapunov function, as is often needed for methods that require $\rho$, see, e.g., \cite{park2022exact}.
The rate of $\rho^k$ matches that of the  Banach--Picard iteration up to a constant factor $C_{\rho}$.
Unlike Theorem~\ref{th:linear_convergence_of_HP}, however, the bound in Theorem~\ref{th:convergence_theorem1} below does not reduce to $\BigOs{1/k}$ as $\rho \to 1$.

\begin{theorem}\label{th:convergence_theorem1}
Let $T : \Hil \to \Hil$ be $\rho$-contractive for some $\rho \in (0, 1)$ and $x^{\star}$ be the unique fixed point of $T$.
Let $\sets{x^k}$ be generated by Algorithm~\ref{alg:adaptive_HP0}.
Then 
\begin{equation}\label{eq:convergence_bound1}
\norms{x^k - T(x^k)}   \leq  C_{\rho} \cdot \rho^k \norms{x^0 - T(x^0)}  \quad \text{and} \quad 
\norms{x^k - x^{\star}} \leq  \frac{C_{\rho}(1+\rho)}{1-\rho} \rho^k \norms{x^0 - x^{\star}},
\end{equation}
where $C_{\rho} := \frac{1 - \rho + 2\rho^2}{(1-\rho)^2}$.
Consequently, $\sets{x^k}$ strongly converges to the unique fixed point $x^{\star}$ of $T$ with a linear rate.
\end{theorem}

\begin{proof}
Set $\lambda_k := \frac{1}{ \bar{\varphi}_k + 1}$, $r^k := x^k - T(x^k)$, and $R_k := \norms{x^k - x^0}$.
First, since $\bar{\varphi}_{k-1} \geq 0$ for all $k \geq 0$, we have $1 + \bar{\varphi}_k = 1 + \bar{\varphi}_{k-1} + \omega^{2(k+1)}_k \geq \omega^{2(k+1)}_k$ for all $k \geq 0$.
Second, by the choice of $\rho_0$ in \eqref{eq:omega_0_choice}, we have $\omega_0 \geq \frac{1}{\rho}$.
Third, if $\omega_k$ is updated by \eqref{eq:update_omega_k}, then we can easily see that $\omega_{k+1} \geq \omega_k \geq \cdots \geq \omega_0 \geq \frac{1}{\rho}$.
Alternatively, if $\omega_k$ is updated by \eqref{eq:update_omega_k_modified}, then, as shown above, $\omega_k \geq \frac{1}{\rho}$.
In both cases, by induction, we have $\omega_k \geq \frac{1}{\rho}$ for all $k \geq 0$.
Thus, we conclude that $1 + \bar{\varphi}_k \geq \rho^{-2(k+1)}$, leading to $\lambda_k := \frac{1}{1+ \bar{\varphi}_k} \leq \rho^{2(k+1)}$ for all $k \geq 0$.

Now, from Step \ref{alg1:step5} of Algorithm~\ref{alg:adaptive_HP0} and $\lambda_k := \frac{1}{\bar{\varphi}_k + 1} \in (0, 1)$, we have $x^{k+1} = \lambda_kx^0 + (1-\lambda_k)T(x^k)$, leading to $x^{k+1} - x^0 = (1-\lambda_k)(T(x^k) - x^0)$.
Using the triangle inequality, this relation, and the $\rho$-contractivity of $T$, we can show that
\begin{equation*}
\arraycolsep=0.2em
\begin{array}{lcl}
R_{k+1} &:= & \norms{x^{k+1} - x^0} = (1-\lambda_k)\norms{T(x^k) - x^0} \leq \norms{T(x^k) - x^0} \vspace{1ex}\\
& \leq & \norms{T(x^k) - T(x^0)} + \norms{x^0 - T(x^0)} \vspace{1ex}\\
& \leq & \rho\norms{x^k - x^0} + \norms{r^0} =  \rho R_k + \norms{r^0}.
\end{array}
\end{equation*}
By induction and $R_0 = 0$, we obtain $R_k \leq \rho^k R_0 + \norms{r^0}\sum_{i=0}^{k-1}\rho^i = \frac{1-\rho^k}{1-\rho}\norms{r^0} \leq \frac{\norms{r^0}}{1-\rho}$.
Using again the $\rho$-contractivity of $T$ and this inequality, we obtain
\begin{equation}\label{eq:th21_proof1}
\arraycolsep=0.2em
\begin{array}{lcl}
\norms{x^0 - T(x^k)} & \leq & \norms{x^0 - T(x^0)} + \norms{T(x^0) - T(x^k)} \leq \norms{r^0} + \rho R_k \vspace{1ex}\\
& \leq & \Big(1 + \frac{\rho}{1-\rho}\Big)\norms{r^0} = \frac{1}{1-\rho}\norms{r^0}.
\end{array}
\end{equation}
Next, using the $\rho$-contractivity of $T$ once again and the triangle inequality, we can derive that
\begin{equation}\label{eq:th21_proof2} 
\hspace{-3ex}
\arraycolsep=0.1em
\begin{array}{lcl}
\norms{r^{k+1}} & = & \norms{x^{k+1} - T(x^{k+1})} \vspace{1ex}\\
& = & \norms{x^{k+1} - T(x^k) + [T(x^k) - T(T(x^k))] + [T(T(x^k)) - T(x^{k+1})]} \vspace{1ex}\\
& \leq & \norms{T(x^k) - T(T(x^k))} + \norms{x^{k+1} - T(x^k)} +  \norms{T(T(x^k)) - T(x^{k+1})} \vspace{1ex}\\
& \leq & \rho\norms{r^k} +  (1 + \rho)\norms{x^{k+1} - T(x^k)}.
\end{array}
\hspace{-3ex}
\end{equation}
On the other hand, again, by Step \ref{alg1:step5} of Algorithm~\ref{alg:adaptive_HP0} and $\lambda_k \in (0, 1)$, we have 
\begin{equation*} 
\arraycolsep=0.2em
\begin{array}{lcl}
\norms{x^{k+1} - T(x^k)} &= & \norms{ \lambda_kx^0 + (1-\lambda_k)T(x^k) - T(x^k)} =  \lambda_k\norms{x^0 - T(x^k)}.
\end{array}
\end{equation*}
Combining this expression, \eqref{eq:th21_proof2}, and \eqref{eq:th21_proof1}, we get 
\begin{equation*} 
\arraycolsep=0.2em
\begin{array}{lcl}
\norms{r^{k+1}} \leq \rho\norms{r^k} + (1+\rho)\lambda_k\norms{x^0 - T(x^k)} \leq \rho\norms{r^k} + \frac{1+\rho}{1-\rho}\lambda_k\norms{r^0}.
\end{array}
\end{equation*}
Since $\lambda_k \leq \rho^{2(k+1)}$, the last inequality leads to
\begin{equation}\label{eq:th21_proof3} 
\arraycolsep=0.2em
\begin{array}{lcl}
\norms{r^{k+1}} \leq \rho\norms{r^k} + \frac{(1+\rho)\norms{r^0}}{1-\rho}   \rho^{2(k+1)} = \rho \norms{r^k} + C_0\norms{r^0} \rho^{2(k+1)},
\end{array}
\end{equation}
where $C_0 : = \frac{1+\rho}{1-\rho}$.
By induction, we can show from \eqref{eq:th21_proof3} that
\begin{equation*} 
\arraycolsep=0.2em
\begin{array}{lcl}
\norms{r^k} & \leq & \rho^k\norms{r^0} + C_0\norms{r^0} \sum_{i=0}^{k-1}\rho^{k-1-i}\rho^{2(i+1)} \vspace{1ex}\\
& = & \rho^k\norms{r^0} + C_0\norms{r^0} \rho^{k+1}\sum_{i=0}^{k-1}\rho^i   =  \Big[ 1 + \frac{C_0\rho(1 - \rho^k)}{1-\rho}\Big]\rho^k \norms{r^0} \vspace{1ex}\\
& \leq & \Big[ 1 + \frac{\rho(1+\rho)}{(1-\rho)^2}\Big] \rho^{k} \norms{r^0}.
\end{array}
\end{equation*}
This is exactly the first bound of \eqref{eq:convergence_bound1} by using $C_{\rho} := 1 + \frac{(1+\rho)\rho}{(1-\rho)^2} = \frac{1 - \rho + 2\rho^2}{(1-\rho)^2}$ and $r^k := x^k - T(x^k)$.

Finally, since $T$ is $\rho$-contractive and $x^{\star} = T(x^{\star})$, we have  
 \begin{equation*}
     \norms{ x^k - x^{\star} }  = \norms{ x^k - T(x^k) + T(x^k) - T(x^{\star}) } \leq \norms{ x^k - T(x^k) } + \rho \norms{ x^k - x^{\star} }.
 \end{equation*}
 Thus, we get $\norms{ x^k - x^{\star} }  \leq \frac{1}{1-\rho} \norms{ x^k - T(x^k)}$. 
 Moreover, by $x^{\star} = T(x^{\star})$ and  the $\rho$-contractivity of $T$, we can easily show that $\norms{r^0} = \norms{x^0 - T(x^0)} \leq \norms{x^0 - x^{\star}} + \norms{T(x^0) - T(x^{\star})} \leq (1+\rho)\norms{x^0 - x^{\star}}$.
Combining the last two inequalities and the first bound of  \eqref{eq:convergence_bound1}, we get the second bound of  \eqref{eq:convergence_bound1}.
Since $\rho^k \to 0$ as $k \to \infty$, the second bound of  \eqref{eq:convergence_bound1} shows that $\sets{x^k}$ converges strongly to $x^{\star} \in \mathrm{Fix}(T)$.
\Eproof
\end{proof}

\beforesec
\section{Main Result 2: Parameter-Free Halpern Method with Line-search}\label{sec:adaptive_Halpern_method}
\aftersec
Inspired by \cite{he2024convergence}, we develop in this section a parameter-free Halpern fixed-point method for approximating the unique fixed point of a $\rho$-contractive mapping. 
Unlike the nonexpansive setting considered in \cite{he2024convergence}, where no problem-dependent parameter is required, our framework naturally involves the contraction parameter $\rho$. 
Our objective is to avoid using the unknown contraction factor $\rho$, but instead to apply a ``line-search" procedure to find a local estimate $\bar{\rho}_k$ for $\rho$ at each iteration $k$. 
Although the resulting algorithm remains \textit{\textbf{parameter-free}}, like Algorithm~\ref{alg:adaptive_HP0}, it may increase the number of $T$ evaluations by at most $\ell$ per iteration (see Lemma~\ref{le:linesearch_steps}).

\beforesubsec
\subsection{\textbf{The derivation of the algorithm}}\label{subsec:A2_derivation}
\aftersubsec
Following the same definitions as in \cite{he2024convergence}, we first define
\begin{equation}\label{eq:residuals}	
	r^k := x^k-T(x^k) \quad \text{and} \quad s^k := x^0-T(x^k).
\end{equation}	
Next, starting from the Halpern fixed-point iteration \eqref{eq:HP_iteration} with $\lambda_k = \hat{\lambda}_k := \frac{1}{\hat{\varphi}_k + 1}$ for a given $\hat{\varphi}_k > 0$ and $u = x^0$, we can rewrite it as
\begin{equation}\label{eq:HP_iteration2}
x^{k+1} := \hat{\lambda}_k x^0 + (1 - \hat{\lambda}_k)T(x^k) = \frac{1}{\hat{\varphi}_k + 1}x^0 + \frac{\hat{\varphi}_k}{\hat{\varphi}_k + 1}T(x^k).
\end{equation}
Using this formula and \eqref{eq:residuals}, we can show that
\begin{equation}\label{eq:a2_derivation1}
\arraycolsep=0.2em
\begin{array}{lcl}
x^{k+1} - T(x^k) & = & \hat{\lambda}_kx^0 + (1-\hat{\lambda}_k)T(x^k) - T(x^k)  =  \hat{\lambda}_k(x^0 - T(x^k)) = \hat{\lambda}_ks^k, \vspace{1ex}\\
x^{k+1} - x^k & = & \hat{\lambda}_k x^0 +  (1-\hat{\lambda}_k)T(x^k) - x^k \vspace{1ex}\\
& = & \hat{\lambda}_k(x^0 - T(x^k))  - (x^k - T(x^k)) = \hat{\lambda}_ks^k - r^k.
\end{array}
\end{equation}
Now, at the $k$-th iteration $(k\geq 0$), let $\bar{\rho}_k \in (0, 1)$ be such that 
\begin{equation}\label{eq:contractive_cond}
\norms{T(x^{k+1}) - T(x^k)} \leq \bar{\rho}_k\norms{x^{k+1} - x^k}.
\end{equation}
Note that, due to the $\rho$-contractivity of $T$, if $\bar{\rho}_k \in [\rho, 1)$, then the condition~\eqref{eq:contractive_cond} automatically holds.
However, since we do not know $\rho$, we will apply a ``line-search" procedure  in the sequel to guarantee \eqref{eq:contractive_cond}.

Utilizing \eqref{eq:contractive_cond}, and the two expressions in \eqref{eq:a2_derivation1}, we can show that 
\begin{equation*}
\arraycolsep=0.2em
\begin{array}{lcl}
0 &\leq & \bar{\rho}_k^2\norms{x^{k+1} - x^k}^2 - \norms{T(x^{k+1}) - T(x^k)}^2 \vspace{1ex}\\
& = & \bar{\rho}_k^2\norms{\hat{\lambda}_ks^k - r^k}^2 - \norms{x^{k+1} - T(x^{k+1}) - (x^{k+1} - T(x^k))}^2 \vspace{1ex}\\
& = & \bar{\rho}_k^2\norms{\hat{\lambda}_ks^k - r^k}^2 - \norms{r^{k+1} - \hat{\lambda}_ks^k}^2 \vspace{1ex}\\
& = & \bar{\rho}_k^2\norms{\hat{\lambda}_ks^k - r^k}^2 - \norms{r^{k+1}}^2 + 2\hat{\lambda}_k\iprods{s^k, r^{k+1}} - \hat{\lambda}_k^2\norms{s^k}^2.
\end{array}
\end{equation*}
This inequality leads to 
\begin{equation*}
\arraycolsep=0.2em
\begin{array}{lcl}
\norms{r^{k+1}}^2 &\leq & 2\hat{\lambda}_k\iprods{s^k,  r^{k+1}} + \bar{\rho}_k^2\norms{\hat{\lambda}_ks^k - r^k}^2  - \hat{\lambda}_k^2\norms{s^k}^2.
\end{array}
\end{equation*}
Therefore, if we impose 
\begin{equation}\label{eq:descent_cond}
\bar{\rho}_k \norms{\hat{\lambda}_ks^k - r^k}  \leq \hat{\lambda}_k\norms{s^k},
\end{equation}
then the last inequality reduces to
\begin{equation}\label{eq:descent_pro}
\norms{r^{k+1} }^2 \leq  2\hat{\lambda}_k\iprods{s^k, r^{k+1}}.
\end{equation}
Our next goals are the following.
First, assuming that $\bar{\rho}_k$ in \eqref{eq:contractive_cond} is given, we develop an adaptive update rule for $\hat{\lambda}_k$ based on Condition \eqref{eq:descent_cond}. 
Next, we apply a ``line-search" procedure to find $\bar{\rho}_k \in (0, 1)$ such that Condition \eqref{eq:contractive_cond} holds.
To achieve our goals, we establish the following four technical lemmas to construct the update rule for $\hat{\lambda}_k$ and to estimate $\bar{\rho}_k$. 
We begin with the following lemma.

\begin{lemma}\label{le:key_properties}
Let $\sets{x^k}_{k \geq 0}$ be updated by \eqref{eq:HP_iteration2} and $\sets{\hat{\varphi}_k} \subset (0, +\infty)$ be such that $\hat{\lambda}_k := \frac{1}{1 + \hat{\varphi}_k}$.
Let $r^k$ and $s^k$ be defined by \eqref{eq:residuals}.
Assume that both conditions \eqref{eq:contractive_cond} and \eqref{eq:descent_cond} are satisfied.
Then, the following statements hold:
\begin{itemize}
\itemsep=0.5em
\item[{\rm(i)}]  $\bar{\rho}_k \norms{x^{k+1} -  x^k} \leq \norms{ x^{k+1} - T(x^k) }$.
\item[{\rm(ii)}] $\norms{ r^{k+1} }^2 \leq 2 \hat{\lambda}_k \iprods{s^k, r^{k+1}}$.
\item[{\rm(iii)}] $\norms{ r^{k+1}}^2 \leq \frac{2}{\hat{\varphi}_k} \iprods{ r^{k+1}, x^0 - x^{k+1}}$.
\end{itemize}
\end{lemma}

\begin{proof}
\noindent{(i).} 
First, from \eqref{eq:residuals} and \eqref{eq:HP_iteration2}, we get $x^{k+1} = T(x^k) + \hat{\lambda}_k s^k$, which immediately implies that $\norms{ x^{k+1} - T(x^k)} = \hat{\lambda}_k \norms{s^k}$ and $x^k - x^{k+1} = (x^k - T(x^k)) - \hat{\lambda}_k s^k = r^k - \hat{\lambda}_k s^k$.
Substituting these two expressions into $\bar{\rho}_k \norms{\hat{\lambda}_ks^k - r^k}  \leq \hat{\lambda}_k\norms{s^k}$ from \eqref{eq:descent_cond}, we obtain $\bar{\rho}_k \norms{ x^k - x^{k+1}} \leq \norms{ x^{k+1} - T(x^k)}$, which proves (i).
		
\vspace{0.75ex}			
\noindent{(ii).} 
Next, utilizing $r^k$ and $s^k$ from \eqref{eq:residuals}, and \eqref{eq:HP_iteration2}, we can show that
\begin{equation*}
r^{k+1} = x^{k+1} - T(x^{k+1}) = T(x^k) - T(x^{k+1}) + \hat{\lambda}_k s^k,
\end{equation*}
which is equivalent to $r^{k+1} - \hat{\lambda}_k s^k = T(x^k) - T(x^{k+1})$.

Now, since $\norms{T(x^{k+1}) - T(x^k)} \leq \bar{\rho}_k \norms{x^{k+1} - x^k}$ due to \eqref{eq:contractive_cond}, using (i), we can deduce that
\begin{equation*}
\norms{ r^{k+1} - \hat{\lambda}_k s^k } = \norms{ T(x^{k+1}) - T(x^k)} \leq \bar{\rho}_k \norms{ x^{k+1} - x^k } \overset{\tiny\mathrm{(i)}}{\leq} \norms{ x^{k+1} - T(x^k) } =  \hat{\lambda}_k  \norms{ s^k }.
\end{equation*}
Squaring both sides and cancelling the common term $\hat{\lambda}_k^2\norms{s^k}^2$, we get (ii).

\vspace{0.75ex}			
\noindent{(iii).} 
Finally, from \eqref{eq:HP_iteration2}, we have $x^0 - x^{k+1} = (1 - \hat{\lambda}_k) (x^0 - T(x^k)) = (1 - \hat{\lambda}_k)s^k$.
Combining this relation and  (ii), we can deduce that $\norms{ r^{k+1} }^2 \leq \frac{2\hat{\lambda}_k}{1 - \hat{\lambda}_k} \iprods{ r^{k+1}, x^0 - x^{k+1}}$.
Since $\hat{\lambda}_k = \frac{1}{\hat{\varphi}_k + 1}$, we obtain (iii) from the last inequality.
\Eproof
\end{proof}

The following result will be used to prove linear convergence of our method.

\begin{lemma}\label{le:varphi_k_bounds}
Let $\sets{x^k}_{k \geq 0}$ be updated by \eqref{eq:HP_iteration2} with $\hat{\lambda}_k = \frac{1}{1 + \hat{\varphi}_k}$ for $\hat{\varphi}_k \in (0, +\infty)$.
Let $r^k$ and $s^k$ be defined by \eqref{eq:residuals}.
Suppose that both conditions \eqref{eq:contractive_cond} and \eqref{eq:descent_cond} hold, and there exists $\nu > 1$ such that  $2\nu \hat{\lambda}_k\iprods{s^k, r^k} \leq \norms{r^k}^2$ for all $k \geq 0$.
Then, we have 
\begin{equation}\label{eq:exponential_bound}
\hat{\varphi}_k \geq \nu^k \left(\hat{\varphi}_0 +  \frac{2\nu - 1}{\nu - 1}\right) -  \frac{2\nu - 1}{\nu - 1}.
\end{equation}
\end{lemma}

\begin{proof}
Note that if $r^k = 0$, then $x^k$ is a fixed point of $T$, and the algorithm is terminated.
We prove \eqref{eq:exponential_bound} for the case $r^k \neq 0$ for all $k \geq 0$.

First, utilizing $s^{k+1}$ and $r^{k+1}$ from \eqref{eq:residuals}, and \eqref{eq:HP_iteration2}, we can easily show that
\begin{equation*}
s^{k+1} = x^0 - T(x^{k+1})  = s^k + T(x^k) - T(x^{k+1})  = r^{k+1} + (1 - \hat{\lambda}_k)s^k.
\end{equation*}	
This relation implies that $\iprods{ r^{k+1}, s^{k+1} } = \norms{ r^{k+1} }^2 + (1 - \hat{\lambda}_k)\iprods{ r^{k+1}, s^k}$.

Next, since $\bar{\rho}_k \norms{\hat{\lambda}_ks^k - r^k}  \leq \hat{\lambda}_k\norms{s^k}$ in \eqref{eq:descent_cond} holds, by Lemma~\ref{le:key_properties}(ii), we get
\begin{equation*}
\iprods{ r^{k+1}, s^{k+1} }  \geq \norms{ r^{k+1} }^2 + \frac{1 - \hat{\lambda}_k}{ 2 \hat{\lambda}_k} \norms{ r^{k+1} }^2 = \frac{1 + \hat{\lambda}_k}{2\hat{\lambda}_k} \norms{ r^{k+1} }^2.
\end{equation*}
For any $k \geq 0$, let us define $U_k := \frac{\iprods{r^k, s^k}}{\norms{r^k}^2}$.
Then, we have $U_{k+1} = \frac{ \iprods{ r^{k+1}, s^{k+1} }}{ \norms{ r^{k+1} }^2 }$ and the last inequality implies that
\begin{equation}\label{eq:B_k}
U_{k+1}	\ge \frac{1 + \hat{\lambda}_k}{2\hat{\lambda}_k}.
\end{equation}
Because $\hat{\lambda}_k \in (0, 1)$, it is obvious  from \eqref{eq:B_k} that $U_{k+1} \geq \frac{1 + \hat{\lambda}_k}{2\hat{\lambda}_k} \geq 1$.
Moreover,  since $s^0 = r^0$, $U_0 := \frac{\iprods{r^0, s^0}}{\norms{r^0}^2} = 1$.
We conclude that $U_k \geq 1$ for all $k \geq 0$.

The condition $2\nu \hat{\lambda}_k\iprods{s^k, r^k} \leq \norms{r^k}^2$ is equivalent to $\hat{\lambda}_kU_k \leq \frac{1}{2\nu}$.
Since it holds for all $k \geq 0$, we have $U_{k+1} \leq \frac{1}{2\nu \hat{\lambda}_{k+1}}$.
Combining this inequality and \eqref{eq:B_k}, we obtain $\frac{1}{\hat{\lambda}_{k+1}} \geq \frac{\nu (1 + \hat{\lambda}_k)}{\hat{\lambda}_k}$.
Substituting $\frac{1}{\hat{\lambda}_k} = \hat{\varphi}_k + 1$ into this inequality, we get
\begin{equation*}
\hat{\varphi}_{k+1} + 1 \geq \nu (\hat{\varphi}_k + 2),
\end{equation*}
which is equivalent to $\hat{\varphi}_{k+1} + \frac{2\nu - 1}{\nu - 1} \geq \nu \left(\hat{\varphi}_k +  \frac{2\nu - 1}{\nu - 1} \right)$.
By induction, we conclude that  $\hat{\varphi}_k \geq \nu^k \big(\hat{\varphi}_0 + \frac{2\nu - 1}{\nu - 1}\big) -  \frac{2\nu - 1}{\nu - 1}$ for $k \geq 0$, which proves \eqref{eq:exponential_bound}.
\Eproof
\end{proof}

The following lemma provides an update rule for $\hat{\lambda}_k$ and other necessary ingredients for our method. 

\begin{lemma}\label{le:implementable_step}
Let $\sets{x^k}_{k \geq 0}$ be updated by \eqref{eq:HP_iteration2} with $\hat{\lambda}_k = \frac{1}{1 + \hat{\varphi}_k}$ for $\hat{\varphi}_k \in (0, +\infty)$.
Let $r^k$ and $s^k$ be defined by \eqref{eq:residuals}.
Suppose that the condition \eqref{eq:contractive_cond} holds.
Then
\begin{itemize}
\itemsep=0.5em
\item[$\mathrm{(i)}$] At iteration $k \geq 0$, if $r^k \neq 0$ and $\bar{\rho}_{k-1}\norms{\hat{\lambda}_{k-1}s^{k-1} - r^{k-1}} \leq \hat{\lambda}_{k-1}\norms{s^{k-1}}$ in \eqref{eq:descent_cond} holds when $k\geq 1$, then we have $s^k \neq 0$ and $\iprods{s^k, r^k} \geq \norms{r^k}^2$.

\item[$\mathrm{(ii)}$] At iteration $k \geq 0$, let $r^k \neq 0$ and $\hat{\lambda}_k$ be updated by
\begin{equation}\label{eq:lambda_k_update}
\arraycolsep=0.2em
\begin{array}{lcl}
\hat{\lambda}_k & := & \frac{\bar{\rho}_k \norms{r^k}^2 }{ \bar{\rho}_k \iprods{r^k, s^k} + \sqrt{ \bar{\rho}_k^2 \iprods{r^k, s^k}^2 + (1 - \bar{\rho}_k^2 )\norms{r^k}^2\norms{s^k}^2} }.
\end{array}
\end{equation}	
Then, we have $\bar{\rho}_k \norms{\hat{\lambda}_ks^k - r^k}  \leq \hat{\lambda}_k\norms{s^k}$ and $\hat{\lambda}_k \in (0, 1)$.

\item[$\mathrm{(iii)}$]
Suppose that $\hat{\lambda}_k$ is updated by \eqref{eq:lambda_k_update} and there exists $\bar{\rho} \in (0, 1)$ such that $0 < \bar{\rho}_k \leq \bar{\rho}$ for all $k \geq 0$. 
Let $\nu := \frac{1 + \bar{\rho} }{2 \bar{\rho} }$.
Then, $\nu > 1$ and $2\nu \hat{\lambda}_k\iprods{s^k, r^k} \leq \norms{r^k}^2$ for all $k \geq 0$.
\end{itemize}
\end{lemma}

\begin{proof}
Throughout, recall from \eqref{eq:HP_iteration2} that $\hat{\lambda}_{k} = \frac{1}{1+\hat{\varphi}_{k}} \in (0, 1)$ whenever $\hat{\varphi}_{k} > 0$.

\vspace{0.5ex}
\noindent{(i).}
If $k = 0$, then $s^0 = x^0 - T(x^0) = r^0 \neq 0$ and $\iprods{s^0, r^0} = \norms{r^0}^2$.
Now, let $k \geq 1$.
From \eqref{eq:residuals} and \eqref{eq:HP_iteration2}, we have 
\begin{equation*}
s^k = x^0 - T(x^k) = (x^0 - x^k) + (x^k - T(x^k)) = (1 - \hat{\lambda}_{k-1} )s^{k-1} + r^k.
\end{equation*}
Since \eqref{eq:contractive_cond} and \eqref{eq:descent_cond} hold at iteration $k-1$, Lemma~\ref{le:key_properties}(ii) with $k-1$ in place of $k$ gives $2\iprods{s^{k-1}, r^k} \geq \frac{1}{\hat{\lambda}_{k-1} }\norms{r^k}^2$.
Using this inequality, $r^k \neq 0$, and $1 - \hat{\lambda}_{k-1}  > 0$, we can show that
\begin{equation*}
\arraycolsep=0.2em
\begin{array}{lcl}
\norms{s^k}^2 & = & (1 - \hat{\lambda}_{k-1}  )^2\norms{s^{k-1}}^2 + 2(1 - \hat{\lambda}_{k-1}  )\iprods{s^{k-1}, r^k} + \norms{r^k}^2 \vspace{1ex}\\
& \geq & (1- \hat{\lambda}_{k-1}  )^2\norms{s^{k-1}}^2 + \frac{1- \hat{\lambda}_{k-1} }{\hat{\lambda}_{k-1} }\norms{r^k}^2 + \norms{r^k}^2 \ \geq \ \frac{1}{\hat{\lambda}_{k-1} }\norms{r^k}^2 \ > \ 0,
\end{array}
\end{equation*}
Hence, we conclude that $s^k \neq 0$.
In the same way, we also get
\begin{equation*}
\begin{array}{l}
\iprods{s^k, r^k} = (1-\hat{\lambda}_{k-1})\iprods{s^{k-1}, r^k} + \norms{r^k}^2 \geq \Big(\frac{1-\hat{\lambda}_{k-1}}{2\hat{\lambda}_{k-1}} + 1\Big)\norms{r^k}^2 = \frac{1+\hat{\lambda}_{k-1}}{2\hat{\lambda}_{k-1}}\norms{r^k}^2 \geq \norms{r^k}^2.
\end{array}
\end{equation*}
Therefore, in both cases $k = 0$ and $k \geq 1$, we have $s^k \neq 0$ and $\iprods{s^k, r^k} \geq \norms{r^k}^2 > 0$.

\vspace{0.75ex}
\noindent{(ii).}
We argue by induction on $k$.
Assume that $r^k \neq 0$ and, if $k \geq 1$, that (ii) holds at iteration $k-1$, i.e., \eqref{eq:descent_cond} holds at iteration $k-1$ with $\hat{\lambda}_{k-1} \in (0, 1)$.
(For $k = 0$ nothing is assumed.)
Then, by (i), we have $s^k \neq 0$ and $\iprods{s^k, r^k} \geq \norms{r^k}^2 > 0$.

For $\lambda \geq 0$, consider the following quadratic function:
\begin{equation*}
\psi_k(\lambda) := \lambda^2\norms{s^k}^2 - \bar{\rho}_k^2\norms{\lambda s^k - r^k}^2
=  (1 - \bar{\rho}_k^2)\norms{s^k}^2\lambda^2 + 2\bar{\rho}_k^2\iprods{s^k, r^k}\lambda - \bar{\rho}_k^2\norms{r^k}^2.
\end{equation*}
Since both sides of \eqref{eq:descent_cond} are nonnegative, condition \eqref{eq:descent_cond} at iteration $k$ is equivalent to $\psi_k(\hat{\lambda}_k) \geq 0$.
Because $\bar{\rho}_k \in (0, 1)$, $s^k \neq 0$, and $r^k \neq 0$, the quadratic function $\psi_k$ has exactly one positive root:
\begin{equation*}
\arraycolsep=0.2em
\begin{array}{lcl}
\lambda_k^{+}  & = & \frac{\bar{\rho}_k\norms{r^k}^2}{\bar{\rho}_k\iprods{r^k, s^k} + \sqrt{\bar{\rho}_k^2\iprods{r^k, s^k}^2 + (1-\bar{\rho}_k^2)\norms{r^k}^2\norms{s^k}^2}}.
\end{array}
\end{equation*}
Thus, $\lambda_k^{+}$ is exactly $\hat{\lambda}_k$ given by \eqref{eq:lambda_k_update}.
Hence $\psi_k(\hat{\lambda}_k) = 0$, i.e., $\bar{\rho}_k\norms{\hat{\lambda}_k s^k - r^k} = \hat{\lambda}_k\norms{s^k}$, so \eqref{eq:descent_cond} holds at iteration $k \geq 0$.

Finally, $\psi_k(0) = -\bar{\rho}_k^2\norms{r^k}^2 < 0$ and, using $\iprods{s^k, r^k} \geq \norms{r^k}^2$ from (i), we have
\begin{equation*}
\psi_k(1) = (1 - \bar{\rho}_k^2)\norms{s^k}^2 + 2\bar{\rho}_k^2\iprods{s^k, r^k} - \bar{\rho}_k^2\norms{r^k}^2 \geq (1 - \bar{\rho}_k^2)\norms{s^k}^2 + \bar{\rho}_k^2\norms{r^k}^2 > 0.
\end{equation*}
Therefore, $\psi_k$ has a root in $(0, 1)$, which must be its unique positive root $\hat{\lambda}_k$.
Hence $\hat{\lambda}_k \in (0, 1)$, or equivalently, $\hat{\varphi}_k = \frac{1}{\hat{\lambda}_k} - 1 > 0$.
This completes the induction step, and (ii) holds for all $k \geq 0$ with $r^k \neq 0$.

\vspace{0.75ex}
\noindent{(iii).}
Since $\iprods{s^k, r^k} \geq \frac{1 + \hat{\lambda}_{k-1}}{2\hat{\lambda}_{k-1}}\norms{r^k}^2 > 0$, if we define $\nu_k := \frac{\norms{r^k}^2}{2\hat{\lambda}_k\iprods{s^k, r^k}}$, then it is well-defined.
Then, by the update rule of $\hat{\lambda}_k$ in \eqref{eq:lambda_k_update}, we can easily show that
\begin{equation*}
\nu_k := \frac{\norms{r^k}^2}{2\hat{\lambda}_k\iprods{s^k, r^k}} \overset{\tiny\eqref{eq:lambda_k_update}}{=} \frac{1}{2} + \frac{\sqrt{ \bar{\rho}_k^2 \iprods{r^k, s^k}^2 + (1 - \bar{\rho}_k^2 )\norms{r^k}^2\norms{s^k}^2} }{2 \bar{\rho}_k \iprods{s^k, r^k}}.
\end{equation*}
Since $\iprods{s^k, r^k} \leq \norms{s^k}\norms{r^k}$ by the Cauchy-Schwarz inequality and  $0 < \bar{\rho}_k \leq \bar{\rho} < 1$, we can show that
\begin{equation*}
\nu_k \geq \frac{1}{2} + \frac{1}{2\bar{\rho}_k} \geq \frac{1}{2} + \frac{1}{2\bar{\rho}}.
\end{equation*}
If we define $\nu := \frac{1 + \bar{\rho}}{2\bar{\rho}}$, then $1 < \nu \leq \nu_k$ for all $k \geq 0$.
This inequality is equivalent to  $2\nu \hat{\lambda}_k\iprods{s^k, r^k} \leq \norms{r^k}^2$ due to the definition of $\nu_k$.
Thus, (iii) is proven.
\Eproof
\end{proof}

Finally, to guarantee the condition \eqref{eq:contractive_cond}, we perform a ``line-search" procedure.
This procedure terminates after a finite number of steps.
First, let us fix $\kappa \in (0, 1)$.

\vspace{0.5ex}
\noindent\textbf{``Line-search" Procedure:} 
Our ``line-search" procedure is described as follows:
\begin{compactitem}
\item[\textit{Initialization:}] Initialize $\bar{\rho}_k := \rho_0$, where $\rho_0$ is computed by \eqref{eq:omega_0_choice}.
\item[\textit{Line-search step:}] 
At each iteration $k \geq 0$, we perform the following steps:
\begin{compactenum}
\item First, check the condition \eqref{eq:contractive_cond}.
If it holds, \textbf{accept} the trial iterate $x^{k+1}$ and \textbf{terminate} this procedure.
Otherwise, \textbf{reject} $x^{k+1}$ and move on.
\item Second, update $\bar{\rho}_k$ as 
\begin{equation}\label{eq:rho_bar_update}
\begin{array}{l}
\bar{\rho}_k := \max\sets{ \tau_k, 1 - \kappa(1 - \bar{\rho}_k) }, \quad \textrm{where} \quad \tau_k := \frac{\norms{T(x^{k+1}) - T(x^k)}}{\norms{x^{k+1} - x^k}}.
\end{array}
\end{equation}
\item Third, recompute $\hat{\lambda}_k$ from \eqref{eq:lambda_k_update} using this new $\bar{\rho}_k$.
\item Finally, recompute $x^{k+1}$ by \eqref{eq:HP_iteration2} using this new $\hat{\lambda}_k$ and go back to \textbf{Step 1}.
\end{compactenum}
\end{compactitem}
The following lemma shows that this ``line-search" procedure terminates after a finite number of steps.
In addition, $\bar{\rho}_k$ is bounded from above by $\bar{\rho} \in (0, 1)$.

\begin{lemma}\label{le:linesearch_steps}
At each iteration $k \geq 0$, the number of rejected trials in the ``line-search" procedure is at most  $\ell := \lceil \log_{1/\kappa}\left( (1 - \rho_0)/(1 - \rho) \right) \rceil$.
Moreover, the parameter $\bar{\rho}_k$ updated by \eqref{eq:rho_bar_update} satisfies
\begin{equation}\label{eq:linesearch_bound}
\rho_0 \leq \bar{\rho}_k \leq \bar{\rho} := 1 - \kappa(1-\rho) \in (\rho, 1).
\end{equation}
\end{lemma}

\begin{proof}
Let us rewrite the update rule \eqref{eq:rho_bar_update} as $\bar{\rho}_{k,j+1} :=  \max\sets{\tau_k, 1 - \kappa(1 - \bar{\rho}_{k,j})}$ for each line-search iteration $j \geq 0$.
This update rule leads to $\bar{\rho}_{k,j+1} \geq 1 - \kappa(1 - \bar{\rho}_{k,j}) \geq \bar{\rho}_{k,j}$.
Therefore, we obtain $1 - \bar{\rho}_{k,j+1} \leq \kappa(1 - \bar{\rho}_{k,j})$ and $\bar{\rho}_{k,j+1} \geq \bar{\rho}_{k,j}$.
By induction and initialization $\bar{\rho}_{k,0} := \rho_0$, we have $1 - \bar{\rho}_{k,j} \leq \kappa^j(1- \bar{\rho}_{k,0}) =  \kappa^j (1 - \rho_0)$.
The line-search procedure is terminated if $\bar{\rho}_{k,j} \geq \rho$, which is satisfied if $\kappa^j(1-\rho_0) \leq 1 - \rho$.
This condition holds if $j \geq \ell := \lceil \log_{1/\kappa}\left( (1-\rho_0)/(1-\rho) \right) \rceil$.
Therefore, we conclude that the line-search procedure terminates after at most $\ell$ iterations.

Note that since $\kappa \in (0, 1)$ and $\rho_0 \leq \rho$, we have $\rho_0 \leq \rho < \bar{\rho} := 1 - \kappa(1 - \rho) < 1$.
Now, we prove \eqref{eq:linesearch_bound} by induction.
\begin{compactitem}[$\bullet$]
\item \textit{Base case:} Since we initialize $\bar{\rho}_k = \bar{\rho}_{k,0} := \rho_0$, we have $\bar{\rho}_k = \rho_0 \leq \rho \leq \bar{\rho}$.
\item \textit{Inductive step:}
At a rejection step, we have $\bar{\rho}_k < \rho$ and  $\bar{\rho}_k$ is replaced by $ \max\sets{ \tau_k, 1 - \kappa(1 - \bar{\rho}_k)}$, which leads to either $\bar{\rho}_k = \tau_k \leq \rho \leq \bar{\rho}$ or $1- \kappa(1- \bar{\rho}_k) \leq 1 - \kappa(1-\rho) = \bar{\rho}$ due to the fact that $\bar{\rho}_k < \rho$.
\end{compactitem}
Overall, we conclude that $\rho_0 \leq \bar{\rho}_k \leq \bar{\rho}$, which is exactly \eqref{eq:linesearch_bound}.
\Eproof
\end{proof}
\noindent\textbf{Warm-start variant.}
The above ``line-search" procedure initializes $\bar{\rho}_k := \rho_0$, where $\rho_0$ is computed by \eqref{eq:omega_0_choice}.
For comparison, one may initialize a search at the previous accepted estimate, i.e., we initialize $\bar{\rho}_k$ at $\bar{\rho}_{k-1}$ for $k \geq 1$ and  at $\rho_0$ only for $k=0$. 
Under global contractivity and the exact acceptance test, every rejection still occurs below $\rho$, and the argument of Lemma~\ref{le:linesearch_steps} preserves the bound $\bar\rho_k \leq \bar{\rho} := 1-\kappa(1-\rho) < 1$. 
Across all iterations, each rejection multiplies $1-\bar\rho$ by at most $\kappa$.
Thus, the same $\ell$ bounds the total number of rejected trials. 

\beforesubsec
\subsection{\textbf{Parameter-free Halpern fixed-point method with line-search}}\label{subsec:PFHFP_linesearch}
\aftersubsec
Combining this ``line-search" procedure, Lemma~\ref{le:implementable_step}, and the scheme \eqref{eq:HP_iteration2}, we can describe the following implementable and parameter-free variant of the adaptive Halpern fixed-point iteration in Algorithm~\ref{alg:adaptive_HP}.

\vspace{-2ex}
\begin{algorithm}[!htbp]
\caption{(Parameter-Free Halpern Fixed-Point Algorithm with Line-search){\!\!}}\label{alg:adaptive_HP}
\begin{algorithmic}[1]
\itemsep=0.2em
\State\textbf{Initialization:} Choose an arbitrary initial point $x^0\in\mathcal H$ and fix $\kappa \in (0, 1)$ (e.g., $\kappa := \frac{1}{2}$).
\State\label{alg2:step2}Estimate an initial value $\rho_0 \in (0, 1)$ as in \eqref{eq:omega_0_choice}.
\State \textbf{For $k = 0, 1, \cdots$, perform}
\State\hspace{3ex}Compute $r^k := x^k - T(x^k)$ and $s^k := x^0 - T(x^k)$.
\Statex\hspace{3ex}If ~$\norms{r^k} = 0$,  then TERMINATE.
\State\hspace{3ex}Initialize $\bar{\rho}_k := \rho_0$.
\State\hspace{3ex}\label{alg2:step7}Compute $\hat{\lambda}_k \in (0, 1)$ as follows: 
\begin{equation*} 
\hat{\lambda}_k := \frac{ \bar{\rho}_k \norms{r^k}^2 }{ \bar{\rho}_k \iprods{r^k, s^k} + \sqrt{\bar{\rho}_k^2 \iprods{r^k, s^k}^2 + (1 - \bar{\rho}_k^2 )\norms{r^k}^2\norms{s^k}^2} }.
\end{equation*}	
\State\hspace{3ex}\label{alg2:step8}Update
\begin{equation*}
x^{k+1} :=  \hat{\lambda}_kx^0 + (1 - \hat{\lambda}_k) T(x^k).
\end{equation*}
\State\hspace{3ex}\label{alg2:step9} If $\norms{T(x^{k+1}) - T(x^k)}  > \bar{\rho}_k\norms{x^{k+1} - x^k}$, then \textbf{reject} $x^{k+1}$, update $\bar{\rho}_k$ by \eqref{eq:rho_bar_update}, 
\Statex\hspace{3ex}and go back to Step~\ref{alg2:step7}.
Otherwise, \textbf{accept} $x^{k+1}$ and proceed to iteration $k+1$.
\State\textbf{End For}
\end{algorithmic}
\end{algorithm}
\vspace{-2ex}

\noindent
Let us clarify several aspects of Algorithm~\ref{alg:adaptive_HP}.
\begin{compactitem}[$\diamond$]
\item In Algorithm~\ref{alg:adaptive_HP}, we incorporate a ``line-search" procedure to guarantee Condition \eqref{eq:contractive_cond} at each iteration $k$.
This procedure requires at most $\ell$ rejected trials and $\ell+1$ tested candidates per iteration $k$, where $\ell := \lceil \log_{1/\kappa}\left((1-\rho_0)/(1-\rho)\right) \rceil$.
Consequently, it requires at most $\ell$ additional evaluations of $T(x^{k+1})$.
In practice, we often choose $\kappa := \frac{1}{2}$, leading to a bisection strategy.

\item Step~\ref{alg2:step7} adds only a negligible computational overhead to the algorithm. 
The main computations consist of one inner product $\iprods{s^k, r^k}$ and two vector norms $\norms{s^k}$ and $\norms{r^k}$, all of which require only $\BigOs{p}$ elementary operations in the $p$-dimensional space, i.e., $\Hil = \R^p$.

\item Step~\ref{alg2:step9} requires the computation of the two norms $\norms{T(x^{k+1}) - T(x^k)}$ and $\norms{x^{k+1} - x^k}$, together with a few basic arithmetic operations. 
\end{compactitem}
Note that Algorithm~\ref{alg:adaptive_HP} applies to the contractive setting of \eqref{eq:FP}, which can be viewed as a complement to the parameter-free Halpern method for nonexpansive mappings, or equivalently, cocoercive equations in \cite{diakonikolas2020halpern}.

\beforesubsec
\subsection{\textbf{Convergence analysis}}\label{subsec:A2_convergence}
\aftersubsec
We establish linear convergence rates of Algorithm~\ref{alg:adaptive_HP} in the following theorem.
	
\begin{theorem}\label{th:convergence_of_adaptive_HP}
Let $ T: \Hil \to \Hil$ be a $\rho$-contractive mapping such that $\rho \in (0, 1)$.
Let  $\sets{x^k}$ be generated by Algorithm \ref{alg:adaptive_HP}. 
Then, for all $k \geq 0$, we have 
\begin{equation}\label{eq:convergence_rate}
\norms{ x^{k+1} - T(x^{k+1}) } \leq \frac{2}{C_k} \norms{x^0 - x^{\star} } \quad \text{and} \quad \norms{x^{k+1} - x^{\star}} \leq \frac{2}{(1-\rho)C_k} \norms{x^0 - x^{\star} },
\end{equation}
where  $x^{\star} \in \mathrm{Fix}(T)$, $\kappa \in (0, 1)$ is fixed, and
\begin{equation*}
\arraycolsep=0.2em
\begin{array}{l}
\bar{\rho}  :=  1 - \kappa(1-\rho) \in (\rho, 1) \quad \text{and} \quad  C_k  :=  \big( \frac{1 + \bar{\rho}}{2\bar{\rho}} \big)^k\big(1 + \frac{2}{1 - \bar{\rho}}\big) - \frac{1 + \bar{\rho}}{1 - \bar{\rho}}.
\end{array}
\end{equation*}
Consequently, $\sets{x^k}$ strongly converges to the unique fixed point $x^{\star}$ of $T$ at a linear rate.
\end{theorem}

\begin{proof}
First, Lemma~\ref{le:linesearch_steps} shows that $\rho_0 \leq \bar{\rho}_k \leq \bar{\rho}$.
Next, by Lemma~\ref{le:implementable_step}(iii), we have $\nu := \frac{1 + \bar{\rho} }{2 \bar{\rho} } > 1$ and $2\nu\hat{\lambda}_k \iprods{s^k, r^k} \leq \norms{r^k}^2$. 
Therefore, both conditions of Lemma~\ref{le:varphi_k_bounds} are satisfied.
Its conclusion holds, i.e., $\hat{\varphi}_k \geq \nu^k \big( \hat{\varphi}_0 + \frac{2\nu - 1}{\nu - 1}\big) -  \frac{2\nu - 1}{\nu - 1}$.
Since $s^0 = r^0$, from \eqref{eq:lambda_k_update}, we obtain $\hat{\lambda}_0 = \frac{\bar{\rho}_0}{1 + \bar{\rho}_0}$, where $\bar{\rho}_0$ denotes the accepted value at $k = 0$.
Because $\hat{\varphi}_0 = \frac{1}{\hat{\lambda}_0} - 1$, we get $\hat{\varphi}_0 = \frac{1}{\bar{\rho}_0}$.
Using $\nu = \frac{1 + \bar{\rho} }{2 \bar{\rho} }$ and $\hat{\varphi}_0 = \frac{1}{ \bar{\rho}_0} \geq 1$, we can show that $\hat{\varphi}_k + 1 \geq \big( \frac{1 + \bar{\rho}}{2\bar{\rho}} \big)^k\big(1 + \frac{2}{1 - \bar{\rho}}\big) - \frac{1 + \bar{\rho}}{1 - \bar{\rho}} =: C_k$.

Next, following the proof in \cite[Theorem 3.2]{he2024convergence}, it is easy to verify that
\begin{equation*} 
\varphi \norms{a}^2 + 2\iprods{a, b} + \norms{ c }^2 -\norms{a + c}^2 = \tfrac{\varphi+1}{2}\norms{ a}^2 - \tfrac{2}{\varphi+1}\norms{a + c - b}^2 + \frac{2}{\varphi+1}\norms{ a + c - b - \tfrac{\varphi+1}{2}a}^2
\end{equation*}
holds for all $a, b, c \in \mathcal{H}$ and an arbitrary positive number $\varphi$. 

Substituting $\varphi := \hat{\varphi}_{k-1}$, $a:= x^k - T(x^k)$, $b:= x^k - x^0$, and $c:=T(x^k) - x^{\star}$, and the facts that $a + c = x^k-x^{\star}$, $a + c - b=x^0-x^{\star}$ into the last equality, we obtain
\begin{equation}\label{eq:th41_proof1}
\arraycolsep=0.2em
\begin{array}{lcl}
\Tc_{[1]} & := & \hat{\varphi}_{k-1} \norms{ x^k - T(x^k) }^2 + 2\iprods{x^k - T(x^k), x^k - x^0 } \vspace{1ex}\\
&& + {~} \norms{T(x^k) - x^{\star}}^2 - \norms{x^k - x^{\star} }^2 \vspace{1ex} \\
& = & \frac{\hat{\varphi}_{k-1} + 1}{2}\norms{ x^k - T(x^k) }^2 - \frac{2}{\hat{\varphi}_{k-1} + 1}\norms{x^0 - x^{\star} }^2 \vspace{1ex}\\
&& + {~} \frac{2}{ \hat{\varphi}_{k-1} + 1}\norms{x^0 - x^{\star} - \frac{\hat{\varphi}_{k-1} + 1}{2}(x^k - T(x^k) )}^2.
\end{array}
\end{equation}
Since $\norms{T(x^k) - x^{\star}} \leq \rho \norms{x^k - x^{\star}}$ and, by Lemma~\ref{le:key_properties}(iii) with $k-1$ in place of $k$, $2\iprods{x^k - T(x^k), x^0 - x^k} \geq  \hat{\varphi}_{k-1}\norms{x^k - T(x^k)}^2$, we get from \eqref{eq:th41_proof1} that $\Tc_{[1]} \leq 0$.
Dropping the last nonnegative term of \eqref{eq:th41_proof1}, we arrive at
\begin{equation*}
\frac{\hat{\varphi}_{k-1} + 1}{2}\norms{x^k - T(x^k)}^2 - \frac{2}{\hat{\varphi}_{k-1} + 1}\norms{ x^0 - x^{\star} }^2 \leq 0,
\end{equation*}
which implies that $\norms{r^k} \leq \frac{2}{\hat{\varphi}_{k-1} + 1}\norms{x^0 - x^{\star}}$ for all $k \geq 1$.
Replacing $k$ by $k+1$ in the last inequality and using $r^{k+1} := x^{k+1} - T(x^{k+1})$ give
\begin{equation*}
\norms{ x^{k+1} - T(x^{k+1})}  \leq \frac{2}{\hat{\varphi}_{k} + 1}\norms{ x^0-x^{\star}}.
\end{equation*}
Substituting  $\hat{\varphi}_k + 1 \geq C_k := \big( \frac{1 + \bar{\rho}}{2\bar{\rho}} \big)^k\big(1 + \frac{2}{1 - \bar{\rho}}\big) - \frac{1 + \bar{\rho}}{1 - \bar{\rho}}$ into the last expression, we obtain the first bound in \eqref{eq:convergence_rate}, where the factor $\frac{1}{C_k}$ in the bound converges geometrically to zero as $k \to \infty$ due to $\frac{1+\bar{\rho}}{2\bar{\rho}} > 1$.

Finally, the proof of the second bound in \eqref{eq:convergence_rate} is similar to the proof of Theorem~\ref{th:convergence_theorem1}, and we omit it here.
\Eproof
\end{proof}

\vspace{-2ex}
\beforesec
\section{Parameter-Free Nesterov-Type Fixed-Point Methods}\label{sec:NesMethods}
\aftersec
As mentioned earlier, when $T$ is $\rho$-contractive, both the Banach--Picard and Krasnosel'ski\v{\i}--Mann fixed-point iterations converge linearly to the unique fixed point of $T$. 
However, in practice, it is often difficult to determine whether the underlying mapping is $\rho$-contractive for some $\rho \in (0,1)$ or merely nonexpansive. 
Consequently, it is generally safer to apply the KM fixed-point iteration, which remains convergent under the weaker assumption of nonexpansiveness.

On the other hand, the KM fixed-point iteration achieves only an $\BigOs{1/\sqrt{k}}$ convergence rate for the fixed-point residual $\norms{x^k - T(x^k)}$ in the nonexpansive setting, which is significantly slower than the optimal $\BigOs{1/k}$ rate attained by Halpern's fixed-point iteration. 
Recently, \cite{attouch2020convergence}, and subsequently \cite{bot2022bfast}, proposed Nesterov-type accelerated variants of the KM fixed-point iteration. 
Interestingly, \cite{tran2022connection} showed that these accelerated KM schemes are mathematically equivalent to Halpern's fixed-point iteration \eqref{eq:HP_iteration}, differing only in the choice of the parameter $\lambda_k$. 
Motivated by this observation, we exploit this equivalence to develop parameter-free Nesterov-type  KM fixed-point methods for approximating a fixed point of $T$.

\beforesubsec
\subsection{\textbf{The derivation of the algorithm}}\label{subsec:Nes_derivation}
\aftersubsec
To derive Nesterov-type accelerated fixed-point schemes for solving \eqref{eq:FP}, we exploit the ideas from \cite{tran2022connection} and proceed as follows.
First, writing \eqref{eq:HP_iteration2} at iterations $k-1$ and $k$ gives
\begin{equation*}
\arraycolsep=0.2em
\begin{array}{l}
(\varphi_{k-1} + 1)x^k  =   x^0 + \varphi_{k-1}T(x^{k-1}) \quad \text{and} \quad (\varphi_k + 1)x^{k+1} =  x^0 + \varphi_kT(x^k).
\end{array}
\end{equation*}
Here, $\varphi_k$ is either $\bar{\varphi}_k$ in Algorithm~\ref{alg:adaptive_HP0} or $\hat{\varphi}_k$ in Algorithm~\ref{alg:adaptive_HP}.

Next, eliminating $x^0$ in both expressions, we can show that $(\varphi_k + 1)x^{k+1} - (\varphi_{k-1} + 1)x^k = \varphi_kT(x^k) - \varphi_{k-1}T(x^{k-1})$.
This equation leads to the following:
\begin{equation}\label{eq:NesFP_scheme0}
\arraycolsep=0.2em
\begin{array}{lcl}
x^{k+1} & := & \frac{\varphi_{k-1} + 1}{\varphi_k + 1}x^k +  \frac{\varphi_k}{\varphi_k + 1}T(x^k) - \frac{\varphi_{k-1}}{\varphi_k + 1}T(x^{k-1}).
\end{array}
\end{equation}
This scheme requires two consecutive iterates $x^{k-1}$ and $x^k$ to update $x^{k+1}$.
Therefore, at the first iteration $k=0$, we update $x^1 := \frac{1}{\varphi_0 + 1}x^0 +  \frac{\varphi_0}{\varphi_0 + 1}T(x^0)$.

Note that we can rewrite \eqref{eq:NesFP_scheme0} as follows:
\begin{equation*} 
\arraycolsep=0.2em
\begin{array}{lcl}
x^{k+1} & := & (1 - \alpha_k)x^k +  \alpha_k T(x^k) + \gamma_k(T(x^k) - T(x^{k-1})),
\end{array}
\end{equation*}
where $\alpha_k := \frac{\varphi_k - \varphi_{k-1}}{\varphi_k + 1}$ and $\gamma_k := \frac{\varphi_{k-1}}{\varphi_k + 1}$.
If $\varphi_k > \varphi_{k-1} \geq 0$, then we have $\alpha_k \in (0, 1)$.
Clearly, if $\gamma_k = 0$, then \eqref{eq:NesFP_scheme0} reduces to the standard KM fixed-point iteration \cite{Bauschke2011}.
Therefore, we can view \eqref{eq:NesFP_scheme0} as a KM fixed-point variant with a \textit{correction term} $\gamma_k(T(x^k) - T(x^{k-1}))$.

Now, we exploit the technique in Algorithm~\ref{alg:adaptive_HP} to construct an update rule for $\varphi_k$ for all $k \geq 0$.
We note that $s^k := x^0 - T(x^k) = (\varphi_{k-1} + 1)x^k - \varphi_{k-1}T(x^{k-1}) - T(x^k) = \varphi_{k-1}(x^k - T(x^{k-1})) + r^k$.
If we denote by $q^k := x^k - T(x^{k-1})$, then we have $s^k = \varphi_{k-1}q^k + r^k$.
The condition $\bar{\rho}_k\norms{\hat{\lambda}_ks^k - r^k} = \hat{\lambda}_k\norms{s^k}$ with $\hat{\lambda}_k = \frac{1}{\varphi_k + 1}$ from Lemma~\ref{le:implementable_step} becomes
\begin{equation*} 
\bar{\rho}_k^2\norms{\varphi_{k-1}q^k - \varphi_kr^k}^2 = \norms{\varphi_{k-1}q^k + r^k}^2.
\end{equation*}
This equation is equivalent the following quadratic equation of $\varphi_k$:
\begin{equation}\label{eq:quadratic_eq2} 
 \bar{\rho}_k^{2}\norms{r^k}^2 \varphi_k^2 - 2 \bar{\rho}_k^{2}\varphi_{k-1}\iprods{q^k, r^k} \varphi_k -\norms{\varphi_{k-1}q^k + r^k}^2 +  \bar{\rho}_k^{2}\varphi_{k-1}^2\norms{q^k}^2 = 0.
\end{equation}
The reduced determinant of this quadratic equation is 
\begin{equation*}
\arraycolsep=0.2em
\begin{array}{lcl}
G_k & := &  \bar{\rho}_k^4 \varphi_{k-1}^2 \iprods{q^k, r^k}^2 + \bar{\rho}_k^2 \norms{r^k}^2\big( \norms{\varphi_{k-1}q^k + r^k}^2 - \bar{\rho}_k^2 \varphi_{k-1}^2 \norms{q^k}^2 \big) \vspace{1ex}\\
& = & \bar{\rho}_k^4\iprods{s^k - r^k, r^k}^2 + \bar{\rho}_k^2\norms{r^k}^2 \big( \norms{s^k}^2 - \bar{\rho}_k^2\norms{s^k - r^k}^2 \big) \vspace{1ex}\\
& = &  \bar{\rho}_k^4\iprods{s^k, r^k}^2 +  \bar{\rho}_k^2(1- \bar{\rho}_k^2)\norms{r^k}^2 \norms{s^k}^2 > 0.
\end{array}
\end{equation*}
Solving the quadratic equation \eqref{eq:quadratic_eq2}, we obtain 
\begin{equation}\label{eq:varphi_k_update}
\varphi_k :=  \frac{ \bar{\rho}_k^2\varphi_{k-1}\iprods{q^k, r^k} + \sqrt{G_k} }{ \bar{\rho}_k^2\norms{r^k}^2 } =  \frac{ \bar{\rho}_k \varphi_{k-1}\iprods{q^k, r^k} + \sqrt{\Delta_k}}{ \bar{\rho}_k \norms{r^k}^2},
\end{equation}
where $\Delta_k := \frac{G_k}{\bar{\rho}_k^2} = \norms{r^k}^2\norms{\varphi_{k-1}q^k + r^k}^2 - \bar{\rho}_k^2\varphi_{k-1}^2\big(\norms{q^k}^2\norms{r^k}^2  - \iprods{q^k, r^k}^2 \big) > 0$.
Here, similar to Algorithm~\ref{alg:adaptive_HP}, we update $\bar{\rho}_k$ by a line-search procedure as in \eqref{eq:rho_bar_update}.

\beforesubsec
\subsection{\textbf{The algorithm and its convergence}}\label{subsec:Nes_convergence}
\aftersubsec
Now, we are ready to present the complete description of the proposed method in Algorithm~\ref{alg:adaptive_NesFP}, which we call the \textit{parameter-free Nesterov-type fixed-point algorithm}.

\begin{algorithm}[!htbp]
\caption{(Parameter-Free Nesterov-Type Fixed-Point Algorithm)}\label{alg:adaptive_NesFP}
\begin{algorithmic}[1]
\itemsep=0.2em
\State\textbf{Initialization:} Choose an arbitrary initial point $x^0\in\mathcal H$.
\Statex Choose $\kappa \in (0, 1)$. 
Set $x^{-1} := x^0$ and $\varphi_{-1} := 0$.
\State \textbf{Option 1:} Estimate $\omega_0$ as in \eqref{eq:omega_0_choice}.
\Statex \textbf{Option 2:} Estimate $\rho_0$ as in \eqref{eq:omega_0_choice}.
\State \textbf{For $k = 0, 1, 2, \cdots$, perform}
\State\hspace{3ex}Compute $r^k := x^k - T(x^k)$ and $q^k := x^k - T(x^{k-1})$.
\Statex\hspace{3ex}If ~$\norms{r^k} = 0$,  then TERMINATE.
\State\hspace{3ex}\textbf{Option 2:} Set $\bar{\rho}_k := \rho_0$.
\State\hspace{3ex}\textbf{Option 1:} Compute $\varphi_{k} := \varphi_{k-1} + \omega_k^{2(k+1)}$.
\Statex\hspace{3ex}\textbf{Option 2:} Compute  
\begin{equation*}
\varphi_k :=   \frac{ \bar{\rho}_k \varphi_{k-1}\iprods{q^k, r^k} + \sqrt{\Delta_k}}{ \bar{\rho}_k \norms{r^k}^2},
\end{equation*}
\Statex\hspace{3ex}where $\Delta_k := \norms{r^k}^2\norms{\varphi_{k-1}q^k + r^k}^2 - \bar{\rho}_k^2\varphi_{k-1}^2\big(\norms{q^k}^2\norms{r^k}^2  - \iprods{q^k, r^k}^2 \big)$.
\State\hspace{3ex}Update
\begin{equation*}
x^{k+1} =  \frac{\varphi_{k-1} + 1}{\varphi_k + 1}x^k +  \frac{\varphi_k}{\varphi_k + 1}T(x^k) - \frac{\varphi_{k-1}}{\varphi_k + 1}T(x^{k-1}).
\end{equation*}
\State\hspace{3ex}\textbf{Option 1:} Update $\omega_{k+1}$ as in Algorithm~\ref{alg:adaptive_HP0}.
\Statex\hspace{3ex}\textbf{Option 2:} Perform the line-search procedure as in Algorithm~\ref{alg:adaptive_HP} to update $\bar{\rho}_k$ by \eqref{eq:rho_bar_update}, \Statex\hspace{3ex}recompute $\varphi_k$ and $x^{k+1}$, and repeat.
\State\textbf{End For}
\end{algorithmic}
\end{algorithm}

\noindent
We make the following observations.
\begin{compactitem}[$\bullet$]
\itemsep=0.2em
\item With accepted operator values cached, each accepted update of \textbf{Option~1} requires one new evaluation of $T$. 
For \textbf{Option~2}, an accepted update preceded by at most $\ell$ rejected trials requires at most $1+\ell$ new evaluations. 
Initial and probe evaluations are counted separately. 
The remaining computations consist of vector additions, scalar-vector multiplications, norm evaluations, and inner products, all of which require $\BigOs{p}$ elementary operations when $\Hil = \R^p$. 

\item The three-term update does not explicitly use the anchor after initialization. In exact arithmetic its invariant is
$(1+\varphi_{k-1})x^k-\varphi_{k-1}T(x^{k-1})=x^0$.
With the same parameter rule, \textbf{Option~1} therefore generates the same iterates as Algorithm~\ref{alg:adaptive_HP0}, and \textbf{Option~2} generates the same iterates as Algorithm~\ref{alg:adaptive_HP}. 
In both case, matching iteration counts between each pair are expected.  
\end{compactitem}
Finally, since the scheme \eqref{eq:NesFP_scheme0} is mathematically equivalent to \eqref{eq:HP_iteration2}, the convergence of Algorithm~\ref{alg:adaptive_NesFP} follows directly from the convergence results established for Algorithms~\ref{alg:adaptive_HP0} and \ref{alg:adaptive_HP}. 
The following corollary formalizes this observation.

\begin{corollary}\label{co:Nes_convergence}
Let $T : \Hil \to \Hil$ be a $\rho$-contractive mapping with $\rho \in (0,1)$. 
Let $\{x^k\}_{k\ge0}$ be generated by Algorithm~\ref{alg:adaptive_NesFP}. 
Then, \textbf{Option 1} of Algorithm~\ref{alg:adaptive_NesFP} inherits the conclusions of Theorem \ref{th:convergence_theorem1}, while \textbf{Option 2} inherits those of Theorem~\ref{th:convergence_of_adaptive_HP}.
\end{corollary}

\beforesec
\section{Applications to Cocoercive Equations}\label{sec:NonExMapping}
\aftersec
In this section, we apply the results developed in Sections~\ref{sec:linear_convergence_rate} and \ref{sec:adaptive_Halpern_method} to solve the cocoercive equation \eqref{eq:CE}. The main idea is summarized as follows.
\begin{compactitem}[$\diamond$]
\itemsep=0.2em
\item First, we apply Tikhonov regularization technique to the operator $G$ in \eqref{eq:CE} by defining $G_{\mu} := G + \mu \Id$ for a sufficiently small parameter $\mu > 0$, where $\Id$ denotes the identity operator.

\item Next, we appropriately choose a parameter $\eta > 0$ and define $T_{\mu} := \Id - \eta G_{\mu}$ so that $T_{\mu}$ becomes a contractive mapping.

\item Then, we apply either Algorithm~\ref{alg:adaptive_HP0} from Section~\ref{sec:linear_convergence_rate} or Algorithm~\ref{alg:adaptive_HP} from Section~\ref{sec:adaptive_Halpern_method} to approximate the fixed point $x^{\star}_{\mu}$ of $T_{\mu}$.

\item Finally, we characterize an approximate solution of \eqref{eq:CE} by choosing an appropriate value of the regularization parameter $\mu$.
\end{compactitem}
\revise{
Notice that the fixed-point problem \eqref{eq:FP} with a nonexpansive mapping $T$ is a special case of \eqref{eq:CE}. 
Indeed, by defining $G := \Id - T$, solving \eqref{eq:FP} is equivalent to solving \eqref{eq:CE}, where $G$ is $\frac{1}{2}$-cocoercive. 
}
Therefore, it suffices to consider the more general problem \eqref{eq:CE} throughout this section.

\beforesubsec
\subsection{\textbf{Tikhonov regularization for \eqref{eq:CE}}}\label{subsec:Tikhonov_reg}
\aftersubsec
Let $G : \Hil \to \Hil$ be a single-valued $\beta$-cocoercive mapping, i.e., $\iprods{G(x) - G(y), x - y} \geq \beta\norms{G(x) - G(y)}^2$ for all $x, y \in \Hil$, and $\mu > 0$ be a given positive number.
We define the regularized mappings $G_{\mu}$ and $T_{\mu}$, respectively, as follows:
\begin{equation}\label{eq:T_mu}
G_{\mu} := G + \mu \Id 
\quad \text{and} \quad T_{\mu} := \Id - \eta G_{\mu} = (1-\eta\mu)\Id - \eta G,
\end{equation}
for an appropriate $\eta > 0$ determined later.
First, we prove the following lemma.

\begin{lemma}\label{le:regularized_mapping}
Let $G : \Hil \to \Hil$ be a $\beta$-cocoercive mapping and $\zer{G} \neq \emptyset$.
Let $G_{\mu}$ and $T_{\mu}$ be defined by \eqref{eq:T_mu} for some $\mu > 0$ and $\eta > 0$.
Then, the following properties hold:
\begin{compactitem}
\item[$\mathrm{(i)}$] $x^{\star}_{\mu} \in \zer{G_{\mu}}$ iff $x^{\star}_{\mu} \in \mathrm{Fix}(T_{\mu})$ for any $\eta > 0$ and $\mu > 0$.
Moreover,  $\zer{G_{\mu}} = \sets{x^{\star}_{\mu}}$, the unique solution.
\item[$\mathrm{(ii)}$] For $x^{\star} \in \zer{G}$ and $x^{\star}_{\mu} \in \zer{G_{\mu}}$, we have 
\begin{equation}\label{eq:G_mu_pro1}
\mu^2\norms{x^{\star}_{\mu} - x^{\star}}^2 + (1 + 2\mu\beta)\norms{G(x^{\star}_{\mu}) - G(x^{\star})}^2 \leq \mu^2\norms{x^{\star}}^2.
\end{equation}
Consequently, we have $\norms{G(x_{\mu}^{\star})} \leq \frac{\mu}{\sqrt{1+2\mu\beta}}\norms{x^{\star}}$ and $\norms{x^{\star}_{\mu} - x^{\star}} \leq \norms{x^{\star}}$.
\item[$\mathrm{(iii)}$] If we choose $\eta$ such that $0 < \eta \leq \frac{2\beta}{1 + 2\mu\beta}$, then $T_{\mu}$ is $\rho_{\mu}$-contractive with $\rho_{\mu} := 1 - \mu\eta \in (0, 1)$.
\end{compactitem}
\end{lemma}

\begin{proof}
\noindent{$\mathrm{(i).}$}
We note that $x^{\star}_{\mu}\in \zer{G_{\mu}}$ means that $G_{\mu}(x^{\star}_{\mu}) = 0$, which is equivalent to $x^{\star}_{\mu} = x^{\star}_{\mu} - \eta G_{\mu}(x^{\star}_{\mu}) = T_{\mu}(x^{\star}_{\mu})$ for any $\mu > 0$ and $\eta > 0$. 
This means that $x^{\star}_{\mu}$ is a fixed point of $T_{\mu}$.
Since $G_{\mu}$ is strongly monotone and continuous, $\zer{G_{\mu}} = \sets{x^{\star}_{\mu}}$.

\vspace{0.75ex}
\noindent{$\mathrm{(ii).}$}
Let $x^{\star} \in \zer{G}$ and $x^{\star}_{\mu} \in \zer{G_{\mu}}$.
Then, we have $-\mu x^{\star} = G(x^{\star}_{\mu} ) - G( x^{\star} ) + \mu(x^{\star}_{\mu} - x^{\star})$.
Squaring both sides, we obtain 
\begin{equation*}
\mu^2\norms{x^{\star}}^2 = \norms{G(x^{\star}_{\mu} ) - G( x^{\star} )}^2 + 2\mu\iprods{G(x^{\star}_{\mu} ) - G( x^{\star} ), x^{\star}_{\mu} - x^{\star}} + \mu^2\norms{x^{\star}_{\mu} - x^{\star}}^2.
\end{equation*}
By the $\beta$-cocoercivity of $G$, this relation leads to \eqref{eq:G_mu_pro1}.
The two remaining bounds are direct consequences of \eqref{eq:G_mu_pro1} and $G(x^{\star}) = 0$.

\vspace{0.75ex}
\noindent{$\mathrm{(iii).}$}
Since $0 < \eta \leq \frac{2\beta}{1 + 2\beta\mu}$, we have $\eta - 2\beta(1 - \eta\mu) \leq 0$.
In addition, since $\frac{2\beta}{1 + 2\beta\mu} < \frac{1}{\mu}$, we get $1 - \eta\mu > 0$.
Now, using these facts, $G_{\mu}$ and $T_{\mu}$ defined by \eqref{eq:T_mu}, and the $\beta$-cocoercivity of $G$, we can show that
\begin{equation*}
\arraycolsep=0.2em
\begin{array}{lcl}
    \norms{ T_{\mu} (x) - T_{\mu} (y) }^2 &= &  \norms{ (1-\eta \mu) (x-y) - \eta (G(x) - G(y)) }^2 \vspace{1ex}\\
    &= & (1-\eta \mu )^2 \norms{ x - y }^2 + \eta^2 \norms{ G(x) - G(y) }^2 \vspace{1ex}\\
    && - {~} 2\eta (1-\eta \mu) \iprods{ G(x) - G(y), x - y } \vspace{1ex} \\    
    &\leq & (1 - \eta \mu )^2 \norms{ x - y }^2 + \eta^2 \norms{ G(x) - G(y) }^2  \vspace{1ex}\\
    && - {~} 2\eta (1-\eta \mu) \beta \norms{ G(x) - G(y) }^2 \vspace{1ex} \\
    & = & (1 - \eta \mu)^2 \norms{ x - y }^2 + \eta [ \eta - 2\beta (1-\eta \mu ) ] \norms{ G(x) - G(y) }^2 \vspace{1ex}\\
    & \leq & (1 - \eta\mu)^2\norms{x - y}^2.
\end{array}
\end{equation*}
Since $1 - \eta\mu > 0$, the last expression leads to $ \norms{ T_{\mu} (x) - T_{\mu} (y) } \leq (1 - \eta\mu)\norms{x - y}$, showing that $T_{\mu}$ is $\rho_{\mu}$-contractive with $\rho_{\mu} := 1 - \eta\mu$.
\Eproof
\end{proof}

Lemma~\ref{le:regularized_mapping} provides useful insights for developing practical algorithms.
First, the bound $\norms{G(x_{\mu}^{\star})} \leq \frac{\mu}{\sqrt{1+2\mu\beta}}\norms{x^{\star}}$ in Lemma~\ref{le:regularized_mapping}(ii) shows that $\norms{G(x^{\star}_{\mu})} = \BigOs{\mu}$ as $\mu \to 0^{+}$.
Second, if the cocoercivity parameter $\beta$ is known, then we may choose $\eta := \frac{2\beta}{1 + 2\beta\mu}$.
With this choice, $T_{\mu}$ depends only on the regularization parameter $\mu$.

\beforesubsec
\subsection{\textbf{Algorithms for solving cocoercive equations}}\label{subsec:methods_for_CE}
\aftersubsec
For simplicity, we present the application of Algorithm~\ref{alg:adaptive_HP0} to approximate a fixed point of $T_{\mu}$. 
The extension to Algorithm~\ref{alg:adaptive_HP} is analogous. 
Once an approximate fixed point of $T_{\mu}$ is obtained, we recover an approximate solution of \eqref{eq:CE} by appropriately choosing the regularization parameter $\mu$.

In fact, by Lemma~\ref{le:regularized_mapping}(ii), if we choose $\mu := \epsilon > 0$ for a sufficiently small tolerance $\epsilon \in (0, 1)$, then $x^{\star}_{\mu} = x^{\star}_{\epsilon}$ and $\norms{G(x^{\star}_{\epsilon})} \leq \frac{\norms{x^{\star}}\epsilon}{\sqrt{1 + 2\beta\epsilon}}$, showing that $x^{\star}_{\epsilon}$ is an $\hat{\epsilon}$-approximate solution of \eqref{eq:CE}, where $\hat{\epsilon} := \frac{\epsilon\norms{x^{\star}} }{\sqrt{1 + 2\beta\epsilon}}$.
Since $\beta$ is often unavailable in practice, for the guaranteed version we instead assume that a known lower bound $\underline{\beta}$ of $\beta$ is given such that  $0<\underline{\beta} \leq \beta$.
Then, we choose $\eta := 2\theta \underline{\beta} /(1+2\underline{\beta} \epsilon)$ for a given $\theta \in(0,1]$.
Monotonicity of $b\mapsto b/(1+2b\epsilon)$ implies the admissibility condition of Lemma~\ref{le:regularized_mapping}(iii).
With these choices of $\mu$ and $\eta$, we specialize Algorithm~\ref{alg:adaptive_HP0} to approximate a fixed point of $T_{\mu}$, or equivalently, an approximate solution of \eqref{eq:CE}, resulting in Algorithm~\ref{alg:AHP_mu}.

Estimating the lower bound $\underline{\beta}$ mentioned above is nontrivial.
In practice, one can perform a heuristic procedure as follows.
First, compute
\begin{equation}\label{eq:beta_0_estimate}
\beta_0 := \frac{\iprods{G(x^0 + v) - G(x^0), v}}{\norms{G(x^0 + v) - G(x^0)}^2},
\end{equation}
for a given direction $0 \neq v\in\Hil$ such that $G(x^0 + v) \neq G(x^0)$.
Since $G$ is $\beta$-cocoercive, we have $\beta_0 \geq \beta$, but hope that it is close to $\beta$.
Next, we set $\underline{\beta} := \sigma\beta_0$ for some small $\sigma \in (0, 1)$, and expect that it is a lower bound of $\beta$.

\begin{algorithm}[!htbp]
\caption{(Adaptive HP Algorithm for Cocoercive Equation \eqref{eq:CE})}\label{alg:AHP_mu}
\begin{algorithmic}[1]
\itemsep=0.2em
\State\textbf{Initialization:} Choose an initial point $x^0\in\mathcal H$ and accuracies $\epsilon, \delta \in (0, 1)$.
\State Choose $\mu := \epsilon$, a valid bound $0<\underline{\beta}\leq\beta$, and $\eta := \frac{2\theta \underline{\beta}}{1 + 2\underline{\beta}\epsilon}$ with $\theta \in(0,1)$.
\State Estimate $\omega_0$ as in \eqref{eq:omega_0_choice} using $T_{\mu}$ in place of $T$ and set $\bar{\varphi}_{-1} := 0$.
\State Evaluate $T_{\mu}(x^0) := (1 - \eta\mu)x^0 - \eta G(x^0)$.
\State \textbf{For $k = 0, 1, 2, \cdots$, perform}
\State\hspace{3ex}If ~$\norms{x^k - T_{\mu}(x^k)} \leq \delta$,  then TERMINATE.
\State\hspace{3ex}Update $\bar{\varphi}_{k} := \bar{\varphi}_{k-1} + \omega_k^{2(k+1)}$.
\State\hspace{3ex}\label{alg3:step5}Update
\begin{equation*}
x^{k+1} :=  \frac{1}{ \bar{\varphi}_k + 1} x^0 + \frac{ \bar{\varphi}_k}{\bar{\varphi}_k + 1}T_{\mu}(x^k).
\end{equation*}
\State\hspace{3ex}\label{alg3:step6}Evaluate $T_{\mu}(x^{k+1}) := (1-\eta\mu)x^{k+1} - \eta G(x^{k+1})$.
\State\hspace{3ex}\label{alg3:step7}Update $\omega_{k+1}$ as in \eqref{eq:update_omega_k} or \eqref{eq:update_omega_k_modified} with $T_{\mu}$ in place of $T$.
\State\textbf{End For}
\end{algorithmic}
\end{algorithm}

Finally, we prove the following result, which states the convergence of Algorithm~\ref{alg:AHP_mu}.
This certified version is not fully parameter-free.
It requires a valid cocoercivity lower bound $\underline{\beta}$ of $\beta$. 
Therefore, in practice we choose $\underline{\beta} := \sigma\beta_0$ for a small $\sigma \in (0, 1)$, e.g., $\sigma := 1/2^4$.

\begin{corollary}\label{co:convergence_of_AHP}
Let $G : \Hil \to \Hil$ in \eqref{eq:CE} be $\beta$-cocoercive and $\mathrm{zer}(G)\neq\emptyset$.
Let $\underline{\beta}$ be such that $0 < \underline{\beta} \leq \beta$.
Suppose that $\mu := \epsilon \in (0, 1)$, $\eta := \frac{2\theta\underline{\beta}}{1 + 2\underline{\beta}\epsilon}$ for some $\theta \in (0, 1]$, and $x^{\star} \in \mathrm{zer}(G)$.
Let $\sets{x^k }_{k\geq 0}$ be generated by Algorithm~\ref{alg:AHP_mu}.
Then, we have 
\begin{equation}\label{eq:AHP_mu_bound} 
\arraycolsep=0.2em
\begin{array}{lcl}
\norms{ G(x^k) }  & \leq & \frac{4(1 - \eta\epsilon)^{k}}{\eta^2\epsilon^2}\big(\norms{G(x^0)} + \epsilon\norms{x^0}\big)  +  2\epsilon\norms{x^{\star}}.
\end{array}
\end{equation}
Consequently, the number of iterations $k$ sufficient to achieve  $\norms{ G(x^k) } \leq M_0\epsilon$ is $k := \BigOs{\epsilon^{-1}\log(\epsilon^{-1})}$, where $M_0 := \frac{4}{\eta^2}\big(\norms{G(x^0)} + \norms{x^0}\big) + 2\norms{x^{\star}}$.
\end{corollary}

\begin{proof}
Note that since $\eta := \frac{2\theta\underline{\beta}}{1 + 2\underline{\beta}\epsilon}$, $0 < \underline{\beta} \leq \beta$, and $\mu = \epsilon$, it satisfies $0 < \eta \leq \frac{2\beta}{1 + 2\beta\mu}$ in Lemma~\ref{le:regularized_mapping}(iii).
Since $x^k$ is generated by Algorithm~\ref{alg:AHP_mu}, by \eqref{eq:convergence_bound1} and $\eta G_{\mu} = \Id - T_{\mu}$, we can derive that
\begin{equation*} 
\arraycolsep=0.2em
\begin{array}{lcl}
\eta\norms{ G_{\mu}(x^k) } & = & \norms{x^k - T_{\mu}(x^k)} \leq  \Big[ 1 + \frac{\rho_{\mu}(1+\rho_{\mu})}{(1-\rho_{\mu})^2}\Big] \rho_{\mu}^k \norms{x^0 - T_{\mu}(x^0)}  \leq  \frac{2\eta \rho_{\mu}^{k} }{(1-\rho_{\mu})^2}\norms{G_{\mu}(x^0)},
\end{array}
\end{equation*}
where we have used the fact that $1 + \frac{\rho_{\mu}(1+\rho_{\mu})}{(1-\rho_{\mu})^2} = \frac{1 - \rho_{\mu} + 2\rho_{\mu}^2}{(1-\rho_{\mu})^2}\leq \frac{2}{(1-\rho_{\mu})^2}$.

Since $\rho_{\mu} := 1 - \eta\mu = 1 - \eta\epsilon$, this expression leads to 
\begin{equation*} 
\norms{ G_{\mu}(x^k) } \leq \frac{2(1 - \eta\epsilon)^{k}}{\eta^2\epsilon^2}\norms{G_{\mu}(x^0)}.
\end{equation*}
Moreover, following the same proof for the second bound of \eqref{eq:convergence_bound1}, we have  $\norms{x^k - x^{\star}_{\mu}} \leq \frac{\eta}{1 - \rho_{\mu}}\norms{G_{\mu}(x^k)}$.
Substituting $\rho_{\mu} = 1 - \eta\mu = 1 - \eta\epsilon$, this inequality leads to
\begin{equation*} 
\norms{ x^k -  x^{\star}_{\mu} } \leq \frac{2(1 - \eta\epsilon)^{k}}{\eta^2\epsilon^3}\norms{G_{\mu}(x^0)}.
\end{equation*}
This inequality together with $\norms{x_{\mu}^{\star} - x^{\star}}  \leq \norms{x^{\star}}$ from Lemma~\ref{le:regularized_mapping}(ii) shows that
\begin{equation*} 
\arraycolsep=0.2em
\begin{array}{lcl}
\norms{ x^k } & \leq & \norms{x^k - x_{\mu}^{\star}} + \norms{x_{\mu}^{\star} - x^{\star}} + \norms{x^{\star}}  \leq \frac{2(1 - \eta\epsilon)^{k}}{\eta^2\epsilon^3}\norms{G_{\mu}(x^0)} + 2\norms{x^{\star}}.
\end{array}
\end{equation*}
Combining the last two inequalities, we can show that
\begin{equation*} 
\arraycolsep=0.2em
\begin{array}{lcl}
\norms{ G(x^k) } & \leq & \frac{2(1 - \eta\epsilon)^{k}}{\eta^2\epsilon^2}\norms{G_{\mu}(x^0)} + \epsilon\norms{x^k} \leq \frac{4(1 - \eta\epsilon)^{k}}{\eta^2\epsilon^2}\norms{G_{\mu}(x^0)}  + 2 \epsilon \norms{x^{\star}} \vspace{1ex}\\
& \leq & \frac{4(1 - \eta\epsilon)^{k}}{\eta^2\epsilon^2 }\big(\norms{G(x^0)} + \epsilon\norms{x^0}\big)  +  2\epsilon\norms{x^{\star}}.
\end{array}
\end{equation*}
For a fixed valid $\underline{\beta}>0$ and $\theta\in(0,1]$, the prescribed $\eta := 2\theta\underline{\beta}/(1+2\underline{\beta}\epsilon)$ is bounded above and bounded away from zero as $\epsilon \downarrow 0$. 
This uniform lower bound is needed for the following complexity statement.

Since we have chosen $\mu = \epsilon \in (0, 1)$, by choosing $k \geq \frac{3\log(\epsilon)}{\log(1 - \eta\epsilon)} = \BigOs{\epsilon^{-1}\log(\epsilon^{-1})}$, we can show that $\frac{(1 - \eta\epsilon)^{k}}{\epsilon^2} \leq \epsilon$.
Therefore, the last expression leads to
\begin{equation*} 
\arraycolsep=0.2em
\begin{array}{lcl}
\norms{ G(x^k) } & \leq & \Big[\frac{4}{\eta^2}\big(\norms{G(x^0)} + \norms{x^0}\big) + 2\norms{x^{\star}}\Big] \epsilon = M_0\epsilon,
\end{array}
\end{equation*}
where $M_0 := \frac{4}{\eta^2}\big(\norms{G(x^0)} + \norms{x^0}\big) + 2\norms{x^{\star}}$.
This proves the complexity statement that the worst-case iteration complexity of Algorithm~\ref{alg:AHP_mu} is $\BigOs{\epsilon^{-1}\log(\epsilon^{-1})}$.
\Eproof
\end{proof}

Note that Algorithm~\ref{alg:AHP_mu} has an iteration complexity of $\BigOs{\epsilon^{-1}\log(\epsilon^{-1})}$.
This bound exceeds the $\BigOs{\epsilon^{-1}}$ complexity of the standard Halpern iteration with $\varphi_k = k + 1$, as stated in Theorem~\ref{th:linear_convergence_of_HP} when $\rho = 1$, by a factor $\log(1/\epsilon)$ in the worst-case.

\begin{remark}\label{re:generalization}
Several important classes of problems can be reformulated as \eqref{eq:CE}, or equivalently as the fixed-point problem \eqref{eq:FP}, including the optimality conditions of convex optimization problems, convex-concave saddle-point problems, and monotone variational inequality problems (VIPs). 
Depending on the underlying problem structure, different reformulations of \eqref{eq:CE} can be employed, such as those based on proximal (or resolvent) operators, forward-backward splitting, backward-forward splitting, Douglas-Rachford splitting, and three-operator splitting. 
Under appropriate assumptions, the resulting reformulation mapping $G$ is cocoercive. Consequently, the algorithms developed in this paper can be directly applied to solve these problems. 
Since these reformulations are well established in the literature, we omit the details and refer the interested reader to the corresponding references, e.g., \cite{Bauschke2011,ryu2016primer}.
\end{remark}

\beforesec
\section{Numerical Experiments}\label{sec:num_examples}
\aftersec
This section reports numerical experiments that evaluate the proposed algorithms and compare them with recent methods from the literature.

\beforesubsec
\subsection{\textbf{Experiment setup}}\label{subsec:experiment_setup}
\aftersubsec
We implement and compare the following seven algorithms.
\begin{compactitem}[$\bullet$]
\item \textbf{Competitor:} The geometric Halpern scheme \eqref{eq:HP_iteration} from \cite{park2022exact} with $\varphi_k := \sum_{i=1}^{k+1}\rho^{-2i}$, denoted by [\texttt{Geo-HP}].

\item \textbf{Competitor:} Algorithm~3.1 from \cite{he2024convergence}, denoted by [\texttt{He-etal24}].

\item \textbf{Our methods:} Two variants of our Algorithm~\ref{alg:adaptive_HP0}, using the update rules \eqref{eq:update_omega_k_modified} and \eqref{eq:update_omega_k}, denoted by [\texttt{Algorithm~\ref{alg:adaptive_HP0}~$(\omega_{\min})$}] and [\texttt{Algorithm~\ref{alg:adaptive_HP0}~$(\omega_{\max})$}], respectively.

\item \textbf{Our method:} Algorithm~\ref{alg:adaptive_HP} with line search, denoted by [\texttt{Algorithm~\ref{alg:adaptive_HP}~(LS)}].

\item \textbf{Our methods:} The two options of Algorithm~\ref{alg:adaptive_NesFP}, denoted by [\texttt{Algorithm~\ref{alg:adaptive_NesFP} (V1)}] --- the Nesterov-type variant of Algorithm~\ref{alg:adaptive_HP0} with the $\omega_{\max}$ update rule, and [\texttt{Algorithm~\ref{alg:adaptive_NesFP} (V2)}]  --- the Nesterov-type variant of Algorithm~\ref{alg:adaptive_HP}.
\end{compactitem}
All algorithms were implemented in Python (Jupyter Notebook) with assistance from OpenAI Codex, and the implementations were carefully reviewed and independently verified by the authors.
All experiments were run on a MacBook Pro with an Apple M4 processor, 24 GB of memory.

In every table, the methods appear in the order \texttt{Geo-HP}, \texttt{He-etal24}, \texttt{Algorithm~\ref{alg:adaptive_HP0} ($\omega_{\min}$)}, \texttt{Algorithm~\ref{alg:adaptive_HP0} ($\omega_{\max}$)}, \texttt{Algorithm~\ref{alg:adaptive_HP}~(LS)}, \texttt{Algorithm~\ref{alg:adaptive_NesFP} (V1)}, and \texttt{Algorithm~\ref{alg:adaptive_NesFP} (V2)}.
Methods not used in a given experiment are omitted.
\begin{compactitem}[$\diamond$]
\item The number of operator evaluations reported in the tables is the number of calls to $T$ per run, counting the initial evaluation, probe evaluations, accepted trial evaluations, and rejected line-search evaluations.

\item All tables and figures below come from a single recorded run, with $50$ instances for Examples 1.1, 1.2 and 2.3, a single instance for Example 2.1, and $300$ random instances of the LASSO problem.

\item Reported solver times are elapsed wall-clock times and exclude data generation, reference-solution construction, spectral norm computation, and post-processing of distance histories.
Since timings depend on the implementation, we regard the number of operator evaluations as the primary measure of cost.
\end{compactitem}
Each algorithm is terminated once the relative fixed-point residual satisfies $\norms{x^k-T(x^k)} \leq \texttt{tol}\cdot\max\left\{1,\norms{x^0-T(x^0)}\right\}$, where $\texttt{tol}$ is a prescribed tolerance.
Unless otherwise specified, the maximum number of iterations is $k_{\max} := 10^4$.

\beforesubsec
\subsection{\textbf{Experiments on \eqref{eq:FP} with contractive mappings}}\label{subsec:example1}
\aftersubsec
We consider two groups of fixed-point problems with contractive mappings: affine mappings in \textbf{Example 1.1} and nonlinear mappings in \textbf{Example 1.2}.

\vspace{2ex}
\noindent\textbf{Example 1.1.} 
We first consider affine mappings $T(x) := \mathbf{Q}x + \mathbf{q}$, where $\mathbf{Q} \in \R^{p\times p}$ and $\mathbf{q}\in\R^{p}$, with $\rho := \norms{\mathbf{Q}}_2 < 1$.
We randomly generate $50$ such mappings $T_i(x)$ with contraction factors $\rho \in [0.5, 0.99]$ and dimension $p = 30$.
The matrix $\mathbf{Q} = \rho\tilde{\mathbf{Q}}/\norms{\tilde{\mathbf{Q}}}_2$ with $\tilde{\mathbf{Q}}$ having i.i.d. $\mathcal{N}(1/\sqrt{p}, 1)$ entries and $\rho\sim\mathcal{U}([0.5, 0.99])$.
The vector $\mathbf{q} \sim \mathcal{N}(0, 0.5\Id_p)$ is sampled independently for each mapping.
The exact solution is available in closed form as $x^{\star} := (\Id - \mathbf{Q})^{-1}\mathbf{q}$, where $\Id$ is the identity matrix.
All instances use the same initial point $x^0 := \mathbf{1}_p$, the vector of all ones.

Table~\ref{tbl:performance_comparison_linear_case_50} summarizes the performance of the seven algorithms with the stopping tolerance $\texttt{tol} = 10^{-8}$.
Here, \textbf{Success} is the number of instances solved to this tolerance; \textbf{Med-Iter} and \textbf{Avg-Iter} are the median and mean numbers of iterations; \textbf{Avg-Eval} is the mean number of evaluations of $T$; \textbf{Avg-Res} is the mean final residual norm; \textbf{Avg-Dist} is the mean of $\norms{x^k - x^{\star}}$; and \textbf{Time[ms]} is the mean solver wall-clock time in milliseconds.

\begin{table}[hpt!]
\vspace{-3ex}
\centering
\caption{Performance comparison of the seven algorithms on $50$ affine contractive mappings.}
\label{tbl:performance_comparison_linear_case_50}
\setlength{\tabcolsep}{5pt}
\renewcommand{\arraystretch}{1.12}
\resizebox{\textwidth}{!}{%
\begin{tabular}{l c c c c c c r}
\toprule
\cell{\textbf{Method}} &  \cell{\textbf{Success}} & \cell{\textbf{Med-Iter}} & \cell{\textbf{Avg-Iter}} & \cell{\textbf{Avg-Eval}} & \cell{\textbf{Avg-Res}} & \cell{\textbf{Avg-Dist}} & \cell{\textbf{Time[ms]}} \\
\midrule
\cell{\texttt{Geo-HP}} & 50 & 28.5 & 54.4 & 55.4 & 4.17e-08 & 7.23e-08 & 1.05 \\
\cell{\texttt{He-etal24}} & 50 & 45.5 & 49.6 & 50.6 & 4.47e-08 & 7.65e-08 & 1.21 \\
\cell{\texttt{Algorithm~\ref{alg:adaptive_HP0} ($\omega_{\min}$)}} & 50 & \pbest{21.0} & 23.2 & 25.2 & 3.76e-08 & 6.24e-08 & 0.71 \\
\cell{\texttt{Algorithm~\ref{alg:adaptive_HP0} ($\omega_{\max}$)}} & 50 & \pbest{21.0} & \pbest{22.9} & \pbest{24.9} & \pbest{3.57e-08} & \pbest{5.83e-08} & \pbest{0.70} \\
\cell{\texttt{Algorithm~\ref{alg:adaptive_HP} (LS)}} & 50 & 33.0 & 37.0 & 44.7 & 4.20e-08 & 7.43e-08 & 1.72 \\
\cell{\texttt{Algorithm~\ref{alg:adaptive_NesFP} (V1)}} & 50 & \pbest{21.0} & \pbest{22.9} & \pbest{24.9} & \pbest{3.57e-08} & \pbest{5.83e-08} & 0.73 \\
\cell{\texttt{Algorithm~\ref{alg:adaptive_NesFP} (V2)}} & 50 & 33.0 & 37.0 & 44.7 & 4.20e-08 & 7.43e-08 & 1.92 \\
\bottomrule
\end{tabular}%
}
\vspace{-4ex}
\end{table}

As Table~\ref{tbl:performance_comparison_linear_case_50} shows, all seven methods solve all $50$ instances, and every proposed variant uses fewer iterations and evaluations of $T$ on average than either baseline.
\texttt{Algorithm~\ref{alg:adaptive_HP0} $(\omega_{\max})$} and \texttt{Algorithm~\ref{alg:adaptive_NesFP} (V1)} attain the lowest mean iteration and evaluation counts.
The two line-search variants average $44.7$ evaluations, compared with $50.6$ for \texttt{He-etal24}, although their mean runtimes are higher.

Figure~\ref{fig:lin_maps_50} gives a complementary view of the seven algorithms.
The left panel shows the distribution of the number of iterations needed to converge, and the right panel shows the final relative residual attained by each algorithm on each of the $50$ mappings.

\begin{figure}[hpt!]
\vspace{-3ex}
	\centering
	\includegraphics[width=\textwidth]{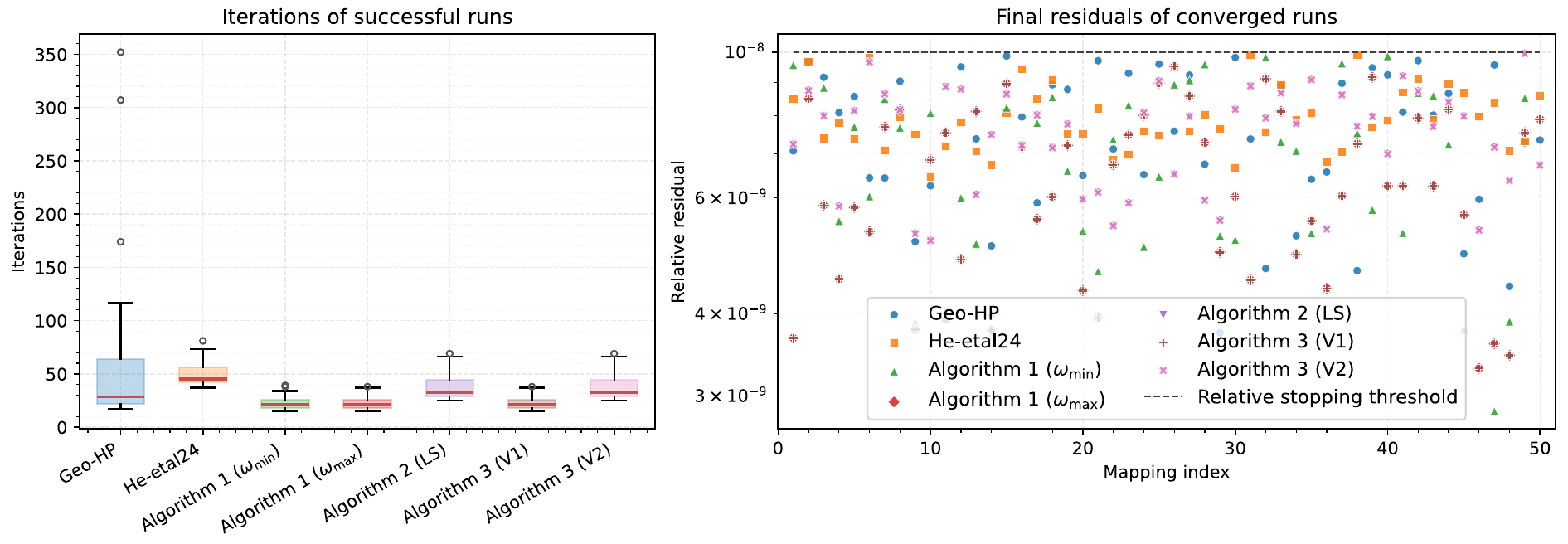}
	\vspace{-2ex}
	\caption{Iteration counts (left) and final relative residuals (right) of the seven algorithms for solving \eqref{eq:FP} with $50$ affine contractive mappings (Example~1.1).}
	\label{fig:lin_maps_50}
	\vspace{-4ex}
\end{figure}

\vspace{2ex}
\noindent\textbf{Example 1.2.} 
Next, we repeat the experiment of \textbf{Example~1.1} with nonlinear mappings of the form $T(x) :=  \alpha\tanh(\mathbf{Q}x + \mathbf{q}) + \mathbf{f}$.
To make $T$ contractive, we first generate $\mathbf{Q} \in \R^{p \times p}$ with i.i.d. $\mathcal{N}(2/\sqrt{p},1)$ entries and rescale it so that $\norms{\mathbf{Q}}_2 \in [0.5, 1.5]$.
We then set $\alpha = \frac{\rho}{\norms{\mathbf{Q}}_2}$ for some $\rho \in [0.5, 0.99]$.
Finally, we independently sample $\mathbf{q}\sim\mathcal{N}(0,\Id_p)$ and $\mathbf{f}\sim\mathcal{N}(\mathbf{1}_p,\Id_p)$.
By construction, $T$ is $\rho$-contractive with $\rho \in [0.5, 0.99]$.

We generate $50$ random mappings with dimension $p=30$ and run the same seven algorithms with the same solver settings as in \textbf{Example~1.1}.
Since $x^{\star}$ is not available in closed form, we compute a reference solution for each mapping by an independent Picard iteration, run until its residual is at most $10^{-14}$.
None of the compared methods is used to produce it.
The resulting a posteriori error bound for the reference solution is its residual divided by $1-\rho$.
Table~\ref{tbl:performance_comparison_nonlinear_case_50} reports the same metrics as Table~\ref{tbl:performance_comparison_linear_case_50}.

\begin{table}[h]
\vspace{-3ex}
\centering
\caption{Performance comparison of the seven algorithms on $50$ nonlinear mappings.}
\label{tbl:performance_comparison_nonlinear_case_50}
\setlength{\tabcolsep}{5pt}
\renewcommand{\arraystretch}{1.12}
\resizebox{\textwidth}{!}{%
\begin{tabular}{l c c c c c c r}
\toprule
\cell{\textbf{Method}} &  \cell{\textbf{Success}} & \cell{\textbf{Med-Iter}} & \cell{\textbf{Avg-Iter}} & \cell{\textbf{Avg-Eval}} & \cell{\textbf{Avg-Res}} & \cell{\textbf{Avg-Dist}} & \cell{\textbf{Time[ms]}} \\
\midrule
\cell{\texttt{Geo-HP}} & 50 & 31.5 & 51.2 & 52.2 & 4.96e-08 & 5.31e-08 & 1.02 \\
\cell{\texttt{He-etal24}} & 50 & 29.0 & 29.3 & 30.3 & 4.53e-08 & 4.86e-08 & 0.74 \\
\cell{\texttt{Algorithm \ref{alg:adaptive_HP0} ($\omega_{\min}$)}} & 50 & \pbest{11.0} & 12.0 & 14.0 & 2.94e-08 & 3.75e-08 & \pbest{0.41} \\
\cell{\texttt{Algorithm \ref{alg:adaptive_HP0} ($\omega_{\max}$)}} & 50 & \pbest{11.0} & \pbest{11.7} & \pbest{13.7} & \pbest{2.70e-08} & \pbest{3.38e-08} & \pbest{0.41} \\
\cell{\texttt{Algorithm \ref{alg:adaptive_HP} (LS)}} & 50 & 15.0 & 16.1 & 18.1 & 3.83e-08 & 4.36e-08 & 0.72 \\
\cell{\texttt{Algorithm \ref{alg:adaptive_NesFP} (V1)}} & 50 & \pbest{11.0} & \pbest{11.7} & \pbest{13.7} & \pbest{2.70e-08} & \pbest{3.38e-08} & 0.42 \\
\cell{\texttt{Algorithm \ref{alg:adaptive_NesFP} (V2)}} & 50 & 15.0 & 16.1 & 18.1 & 3.83e-08 & 4.36e-08 & 0.79 \\
\bottomrule
\end{tabular}%
}
\vspace{-2ex}
\end{table}

Table~\ref{tbl:performance_comparison_nonlinear_case_50} shows that all seven methods solve every instance; \texttt{Algorithm~\ref{alg:adaptive_HP0} $(\omega_{\max})$} and \texttt{Algorithm~\ref{alg:adaptive_NesFP} (V1)} attain the lowest mean iteration and evaluation counts.
Both line-search variants average $18.1$ evaluations versus $30.3$ for \texttt{He-etal24}, but only \texttt{Algorithm~\ref{alg:adaptive_HP} (LS)} has a lower mean runtime.
Figure~\ref{fig:nonlin_maps_50} shows the corresponding distributions: the number of iterations needed to converge (left) and the final relative residual of each algorithm on each of the $50$ mappings (right).

\begin{figure}[hpt!]
	\vspace{-3ex}
	\centering
	\includegraphics[width=\textwidth]{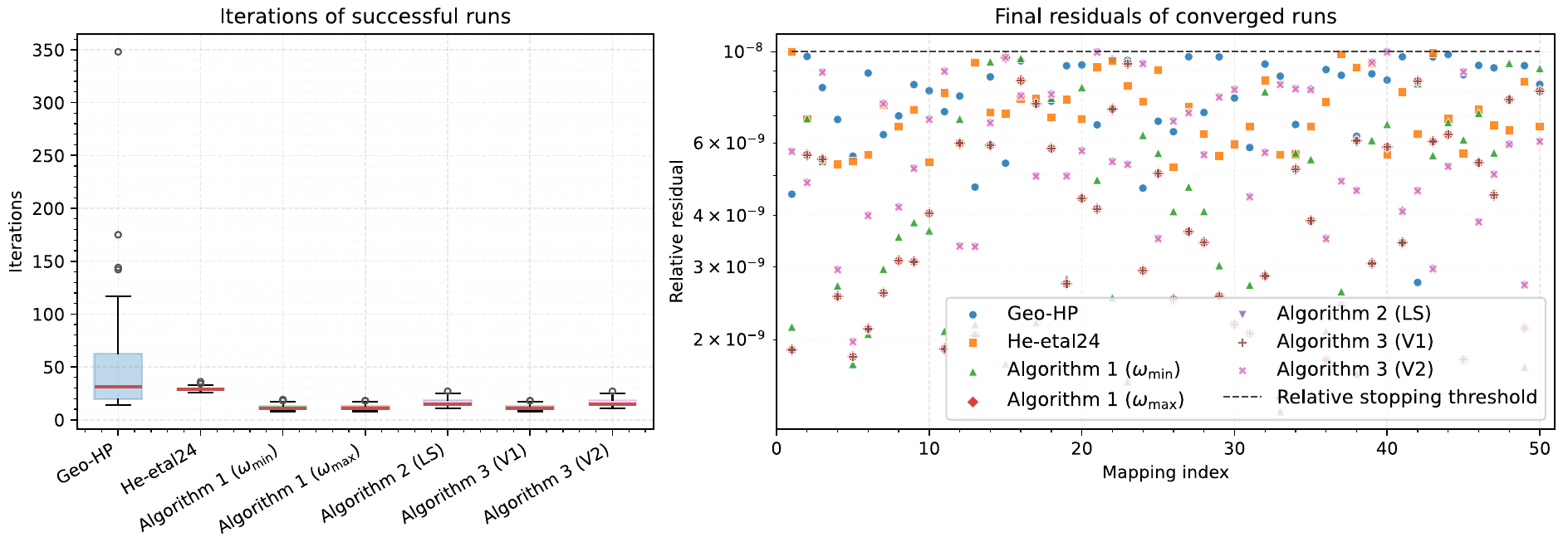}
	\vspace{-2ex}
	\caption{Iteration counts (left) and final relative residuals (right) of the seven algorithms for solving \eqref{eq:FP} with $50$ nonlinear contractive mappings (Example~1.2).}
	\label{fig:nonlin_maps_50}
	\vspace{-4ex}
\end{figure}

\beforesubsec
\subsection{\textbf{Experiments on \eqref{eq:FP} with nonexpansive mappings}}\label{subsec:exam2}
\aftersubsec
We now turn to fixed-point problems with nonexpansive mappings.
We compare the five proposed algorithms with the adaptive Halpern method of \cite{he2024convergence} and with the classical Halpern iteration using the anchoring parameter $\lambda_k = \frac{1}{k+2}$.
Since the latter is the method of \cite{park2022exact} with $\rho = 1$, we continue to denote it by \texttt{Geo-HP}.

\vspace{2ex}
\noindent\textbf{Example 2.1.}
We consider the following nonexpansive mapping $T : \R^3 \to \R^3$, studied in \cite{he2021optimal,he2024convergence}:
\begin{equation}\label{eq:T_exam3}
T(x) := \frac{1}{54.5}\begin{pmatrix}
-35x_1 - \sqrt{|x_1|+1} - 10x_2 + 14x_3 + 1\\
-10x_1 - 26x_2 - 0.5\sin(x_2) + 4x_3\\
14x_1 + 4x_2 - 38x_3 - \arctan(0.5x_3)
\end{pmatrix}.
\end{equation}
As shown in \cite{he2021optimal,he2024convergence}, $T$ is nonexpansive, and its fixed point is $x^{\star} = (0, 0, 0)^{\top}$.

We apply the seven algorithms to this mapping from the initial point $x^0 = (1, 1, 1)^{\top}$, with at most $50,\!000$ iterations and the stopping tolerance $\mathtt{tol} = 10^{-8}$ on the relative residual.
These settings are our own: Example~5.1 in  \cite{he2024convergence} specifies the mapping but not the initial point, tolerance, or iteration budget, so we do not claim to reproduce their reported run exactly.
Figure~\ref{fig:Exam51} plots the relative fixed-point residual (left) and the distance to the solution $\norms{x^k - x^{\star}}$ (right) against the iteration count.

\begin{figure}[hpt!]
\begin{center}
	\includegraphics[width=\textwidth]{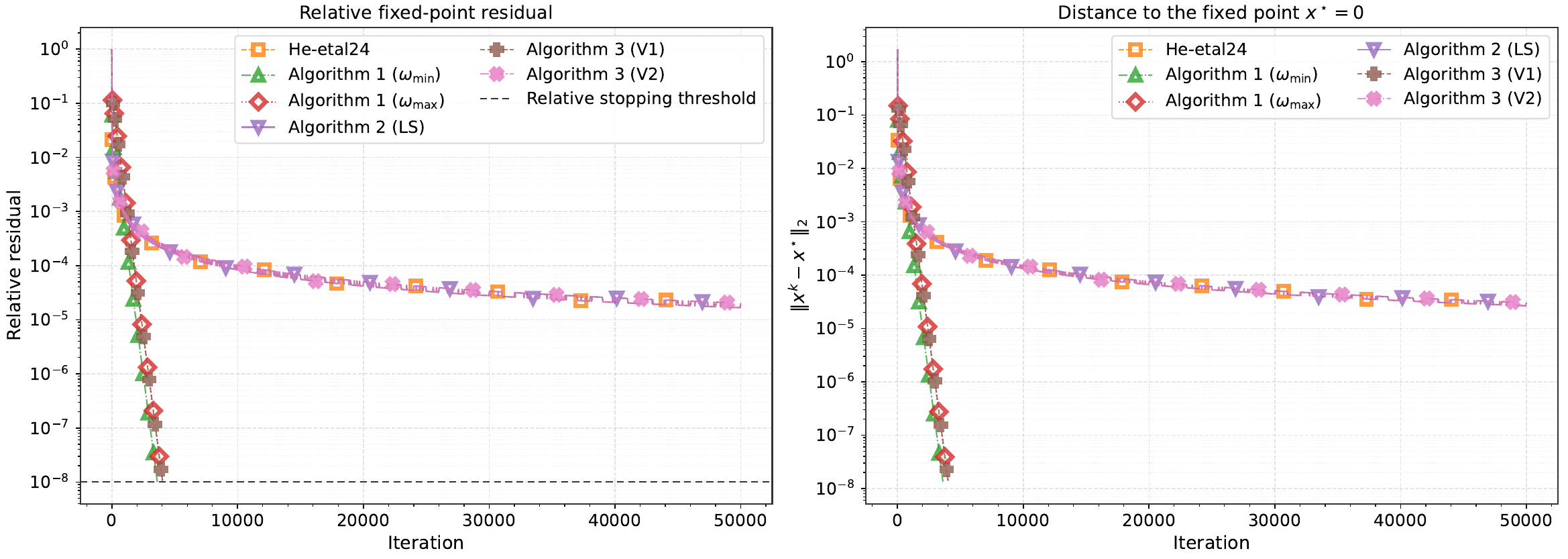}
	\vspace{-2ex}
	\caption{Relative fixed-point residual (left) and distance to the solution (right) versus iteration count for six algorithms on the mapping \eqref{eq:T_exam3} (Example~2.1).}
	\label{fig:Exam51}
\end{center}	
\vspace{-4ex}
\end{figure}

All seven methods are applied to the original mapping $T$, without averaging.
\texttt{Algorithm~\ref{alg:adaptive_HP0}~($\omega_{\min}$)} converges in $3619$ iterations, and \texttt{Algorithm~\ref{alg:adaptive_HP0}~($\omega_{\max}$)} and \texttt{Algorithm~\ref{alg:adaptive_NesFP}~(V1)} converge in $4042$ iterations.
\texttt{Geo-HP}, \texttt{He-etal24}, and the two line-search variants reach the $50,000$-iteration limit without meeting the tolerance; we report them as failures. 
Figure~\ref{fig:Exam51} shows \texttt{He-etal24}, both variants of Algorithm~\ref{alg:adaptive_HP0}, Algorithm~\ref{alg:adaptive_HP}, and both options of Algorithm~\ref{alg:adaptive_NesFP}; \texttt{Geo-HP} is omitted from the plot.

\vspace{2ex}
\noindent\textbf{Example 2.2.} 
Following \cite{he2024convergence}, we test the algorithms on the well-known LASSO (Least Absolute Shrinkage and Selection Operator) problem:
\begin{equation}\label{eq:lasso}
\min_{x \in \R^p} \Big\{\Lc(x) := \tfrac{1}{2m}\norms{Ax - b}^2_2 + \tau \norms{x}_1 \Big\},
\end{equation}
where $A \in \R^{m\times p}$ and $b \in \R^m$ are given input data, $\norms{\cdot}_1$ is the $\ell_1$-norm, and $\tau > 0$ is a regularization parameter.

For $0 < \gamma < \frac{2m}{\norms{A}_2^2}$, the mapping $T(x) : = \mathrm{prox}_{\gamma\tau\norms{\cdot}_1}\big(x - \frac{\gamma}{m} A^{\top}(Ax - b) \big)$ is nonexpansive, where $\mathrm{prox}_f$ denotes the proximal operator of $f$.
Moreover, $x^{\star}$ solves \eqref{eq:lasso} if and only if it is a fixed point of $T$.

We use $10$ LASSO problem sizes, with $p = 512j$ for $j = 1,\ldots,10$ and $m = \lceil0.6p\rceil$, and generate $30$ random instances of each size.
The design matrix $A$ is drawn from a correlated Gaussian model with first-order autoregressive (AR(1)) covariance and correlation coefficient $0.3$, and each column is centered and normalized to have $\ell_2$-norm $\sqrt{m}$.
The ground-truth signal $x^{\sharp}$ has $\lceil0.05p\rceil$ nonzero entries, independently drawn as $\operatorname{sign}(Z)(1+|Z|)$ with $Z\sim\mathcal N(0,1)$.
The observation vector is $b := Ax^{\sharp} + \varepsilon$, where the noise $\varepsilon$ is scaled to give a signal-to-noise ratio of $25$ dB.
For each instance, we set $\tau = 0.05\,\tau_{\max}$, where $\tau_{\max}=\frac{\norms{A^\top b}_\infty}{m}$ is the smallest value of $\tau$ for which the zero vector solves \eqref{eq:lasso}.

We run the six adaptive algorithms on all $300$ LASSO instances.
We omit \texttt{Geo-HP}, which did not reach the tolerance within the iteration budget in preliminary runs.
Every run starts from $x^0 = \mathbf{1}_p$ and stops at relative residual tolerance $\mathtt{tol}=10^{-6}$ or after $5000$ iterations.
For the line-search variants, we reuse the last accepted contraction estimate, permit the endpoint $\rho=1$, and include floating-point roundoff slack in the acceptance test.

The stepsize is $\gamma = \frac{1.9m}{ 1.01\norms{A}_2^2}$.
Table~\ref{tbl:aggregate_performance_lasso} reports the aggregate performance of the six algorithms.
Here, \textbf{Med-Res} is the median relative residual, normalized by $\max\{1,\norms{x^0-T(x^0)}\}$ as in the stopping test; \textbf{Avg-Res} and \textbf{Max-Res} are the mean and maximum absolute residual norms; \textbf{ObjVal} is the mean objective value; and \textbf{Time[s]} is the mean solver wall-clock time in seconds.

\begin{table}[hpt!]
\centering
\caption{Overall performance of the six algorithms on $300$ LASSO instances ($30$ random instances for each of $10$ problem sizes).}
\label{tbl:aggregate_performance_lasso}
\setlength{\tabcolsep}{5pt}
\renewcommand{\arraystretch}{1.12}
\resizebox{\textwidth}{!}{%
\begin{tabular}{l c c c c c c c c r}
\toprule
\cell{\textbf{Method}}
& \cell{\textbf{Success}}
& \cell{\textbf{Med-Iter}}
& \cell{\textbf{Avg-Iter}}
& \cell{\textbf{Avg-Eval}}
& \cell{\textbf{Med-Res}}
& \cell{\textbf{Avg-Res}}
& \cell{\textbf{Max-Res}}
& \cell{\textbf{ObjVal}} {\!\!\!}
& \cell{{\!\!\!}\textbf{Time[s]}} \\
\midrule
\cell{\texttt{He-etal24}}
& 300/300 & 100.0 & 100.2 & 101.2 & 9.37e-07 & 3.97e-05 & 6.00e-05 & 52.34846 & 0.1381 \\
\cell{\texttt{Algorithm~\ref{alg:adaptive_HP0} ($\omega_{\min}$)}}
& 300/300 & 84.0 & 86.8 & 88.8 & 9.30e-07 & 3.93e-05 & 5.95e-05 & 52.34824 & 0.1344 \\
\cell{\texttt{Algorithm~\ref{alg:adaptive_HP0} ($\omega_{\max}$)}}
& 300/300 & \pbest{79.0} & \pbest{79.1} & \pbest{81.1} & \pbest{9.22e-07} & \pbest{3.92e-05} & \pbest{5.92e-05} & \pbest{52.34816} & \pbest{0.1335} \\
\cell{\texttt{Algorithm~\ref{alg:adaptive_HP} (LS)}}
& 300/300 & 98.0 & 98.6 & 103.3 & 9.32e-07 & 3.96e-05 & 5.95e-05 & 52.34845 & 0.1790 \\
\cell{\texttt{Algorithm~\ref{alg:adaptive_NesFP} (V1)}}
& 300/300 & \pbest{79.0} & \pbest{79.1} & \pbest{81.1} & \pbest{9.22e-07} & \pbest{3.92e-05} & \pbest{5.92e-05} & \pbest{52.34816} & 0.1450 \\
\cell{\texttt{Algorithm~\ref{alg:adaptive_NesFP} (V2)}}
& 300/300 & 98.0 & 98.6 & 103.3 & 9.32e-07 & 3.96e-05 & 5.95e-05 & 52.34845 & 0.1925 \\
\bottomrule
\end{tabular}%
}
\end{table}

\begin{figure}[hpt!]
	\centering
	\includegraphics[width=\textwidth]{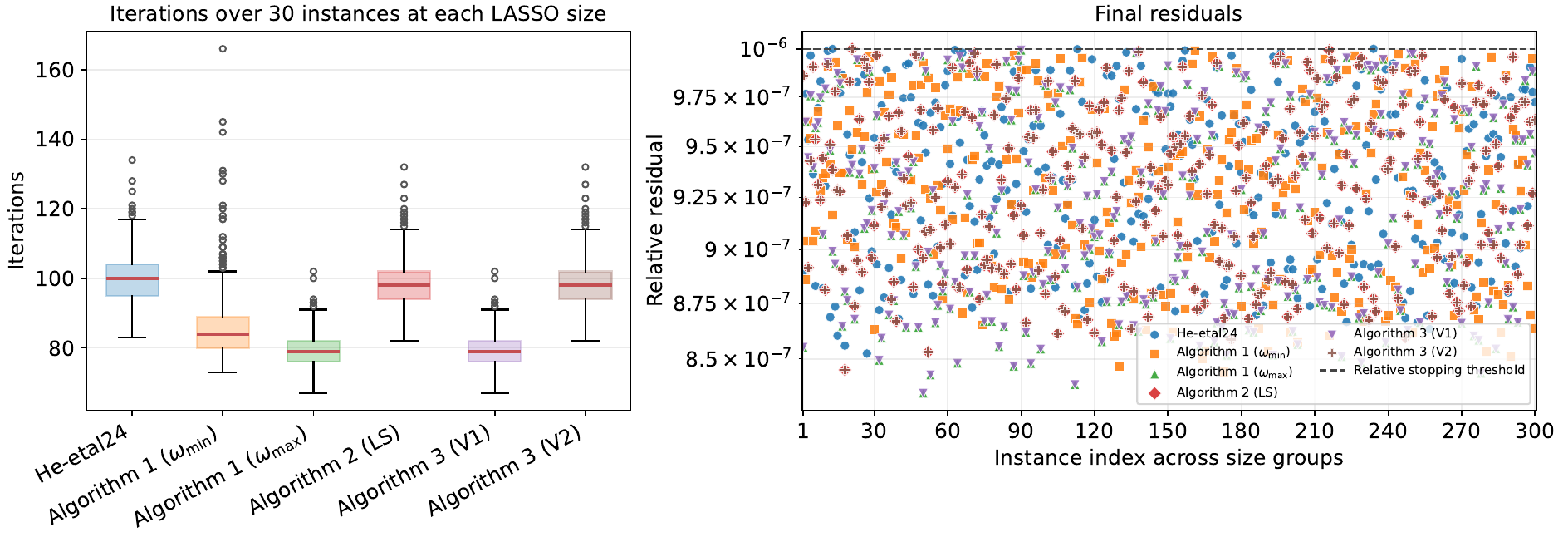}
	\vspace{-3ex}
	\caption{Distributions of iteration counts $($left$)$ and final relative residuals $($right$)$ of six algorithms on $300$ instances of the LASSO problem \eqref{eq:lasso} (Example~2.2).}
	\label{fig:aggregate_performance_lasso}
	\vspace{-3ex}
\end{figure}

As Table~\ref{tbl:aggregate_performance_lasso} shows, all six methods reach the relative residual tolerance of $10^{-6}$ on all $300$ LASSO instances.
\texttt{Algorithm~\ref{alg:adaptive_HP0} ($\omega_{\max}$)} and \texttt{Algorithm~\ref{alg:adaptive_NesFP} (V1)} attain the lowest iteration and evaluation counts on average.
Each line-search variant rejects $804$ trials and averages fewer iterations than \texttt{He-etal24} ($98.59$ versus $100.19$), but more evaluations of $T$ ($103.27$ versus $101.19$) and a higher runtime.
Figure~\ref{fig:aggregate_performance_lasso} shows the distributions of iteration counts and final residuals over all runs.

\vspace{2ex}
\noindent\textbf{Example 2.3.} 
Our last experiment uses $50$ nonexpansive mappings of the form
\begin{equation*}
T(x) := \frac{1}{2}\big( \mathrm{proj}_C(x) + Q\tanh(x) \big), \ \textrm{where} \ C := [-1, 1]^p \ \textrm{and} \ Q^{\top}Q = \Id_p.
\end{equation*}
Here, $\mathrm{proj}_C$ denotes the projection onto the closed convex set $C$.
It is easy to verify that $T$ is nonexpansive, but it need not be globally contractive (e.g., when $Q$ has an eigenvalue $1$).
Moreover, $x^{\star} = \mathbf{0}_p$ is the unique fixed point of $T$ since $\norms{T(x)} < \norms{x}$ for $x \neq 0$.

We generate $50$ such mappings with $p=50$ by sampling $g\sim\mathcal{N}(0,\Id_p)$, normalizing $u=g/\norms{g}_2$, and setting $Q=2uu^{\top}-\Id_p$, so $Qu=u$.
We run the seven algorithms on these instances from the initial point $x^0 = \mathbf{1}_p$, with at most $30,\!000$ iterations and the tolerance $\mathtt{tol} = 10^{-5}$.

Table~\ref{tbl:performance_comparison_nonexpansive_mapping_50} summarizes the results using the same metrics as Table~\ref{tbl:performance_comparison_linear_case_50}.
Figure~\ref{fig:nonlin_nonexpansive_maps_50} complements the table with the distributions of iteration counts (left) and final relative residuals (right) for the algorithms that converged.

\begin{table}[h]
\centering
\caption{Performance comparison of the seven algorithms on $50$ nonlinear nonexpansive mappings of dimension $50$.}
\label{tbl:performance_comparison_nonexpansive_mapping_50}
\setlength{\tabcolsep}{5pt}
\renewcommand{\arraystretch}{1.12}
\resizebox{\textwidth}{!}{%
\begin{tabular}{l c c c c c c r}
\toprule
\cell{\textbf{Method}} {\!\!\!\!} &  \cell{\textbf{Success}} & \cell{\textbf{Med-Iter}} & \cell{\textbf{Avg-Iter}} & \cell{\textbf{Avg-Eval}} & \cell{\textbf{Avg-Res}} & \cell{\textbf{Avg-Dist}} & \cell{\textbf{Time[ms]}} \\
\midrule
\cell{\texttt{Geo-HP}} {\!\!\!} & 0 & 30000.0 & 30000.0 & 30001.0 & 2.35e-04 & \pbest{1.27e-01} & 630.76 \\
\cell{\texttt{He-etal24}} {\!\!\!} & 50 & 1879.0 & 1551.1 & 1552.1 & 5.80e-05 & 1.55e-01 & 40.89 \\
\cell{\texttt{Algorithm~\ref{alg:adaptive_HP0} ($\omega_{\min}$)}} {\!\!\!} & 50 & 1442.0 & 1918.3 & 1920.3 & \pbest{5.49e-05} & 1.51e-01 & 59.46 \\
\cell{\texttt{Algorithm~\ref{alg:adaptive_HP0} ($\omega_{\max}$)}} {\!\!\!} & 50 & \pbest{1400.0} & \pbest{1106.2} & \pbest{1108.2} & 6.01e-05 & 1.70e-01 & \pbest{34.59} \\
\cell{\texttt{Algorithm~\ref{alg:adaptive_HP} (LS)}} {\!\!\!} & 50 & 1871.5 & 1540.8 & 1550.5 & 5.78e-05 & 1.55e-01 & 94.85 \\
\cell{\texttt{Algorithm~\ref{alg:adaptive_NesFP} (V1)}} {\!\!\!} & 50 & \pbest{1400.0} & \pbest{1106.2} & \pbest{1108.2} & 6.01e-05 & 1.70e-01 & 36.22 \\
\cell{\texttt{Algorithm~\ref{alg:adaptive_NesFP} (V2)}} {\!\!\!} & 50 & 1871.5 & 1540.8 & 1550.5 & 5.78e-05 & 1.55e-01 & 101.84 \\
\bottomrule
\end{tabular}%
}
\end{table}

\begin{figure}[hpt!]
	\centering
	\includegraphics[width=\textwidth]{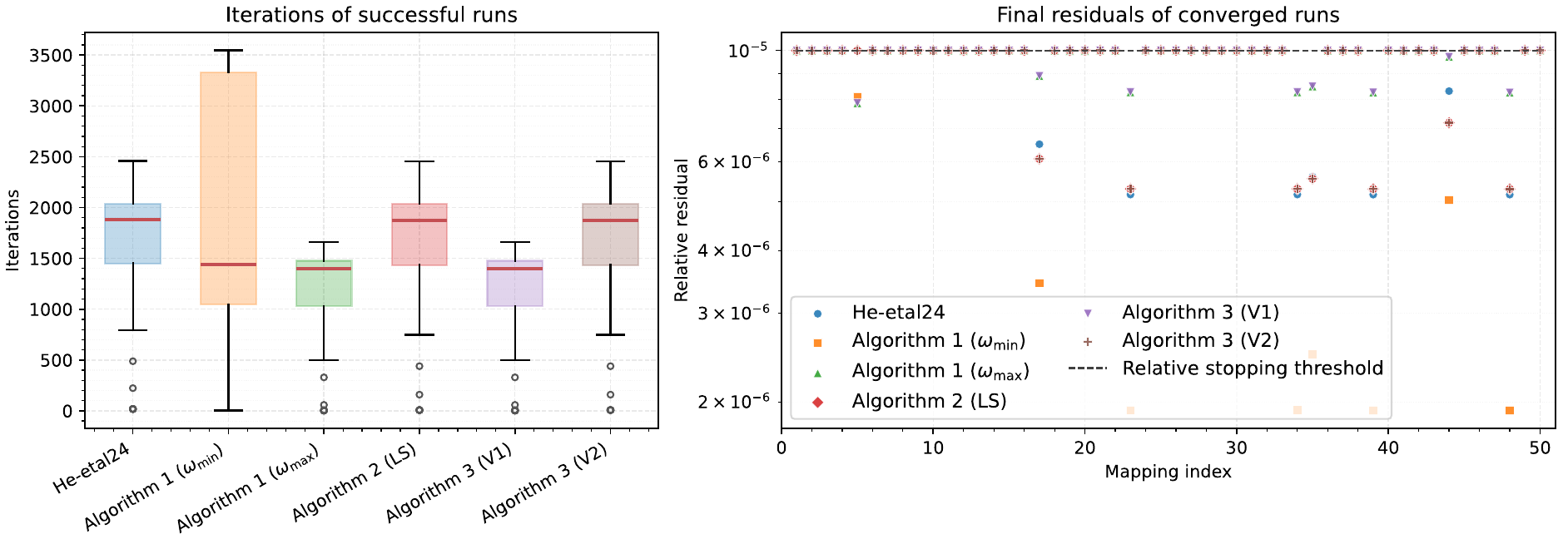}
	\vspace{-2ex}
	\caption{Iteration counts (left) and final relative residuals (right) of the six convergent algorithms for solving \eqref{eq:FP} with $50$ nonlinear nonexpansive mappings (Example~2.3).}
	\label{fig:nonlin_nonexpansive_maps_50}
	\vspace{-3ex}
\end{figure}

As Table~\ref{tbl:performance_comparison_nonexpansive_mapping_50} shows, \texttt{Geo-HP} reaches the $30{,}000$-iteration limit on all instances, whereas \texttt{He-etal24} and all proposed variants attain the tolerance on all $50$ instances.
\texttt{Algorithm~\ref{alg:adaptive_HP0} $(\omega_{\max})$} and \texttt{Algorithm~\ref{alg:adaptive_NesFP} (V1)} attain the lowest mean iteration and evaluation counts, while the $\omega_{\min}$ variant requires more of both than \texttt{He-etal24}.
The line-search variants offer only a small reduction in mean iterations ($1540.8$ versus $1551.1$) and evaluations ($1550.5$ versus $1552.1$), with substantially higher mean runtimes.
Near the solution $x^{\star} = \mathbf{0}_p$, $T$ behaves like $\frac{1}{2}(\Id + Q)$, which has eigenvalue $1$ along $u$, so small residuals do not imply small errors for this family. 
Therefore, the distance to solution error \textbf{Avg-Dist} remains of order $10^{-1}$.

These results describe performance on this particular reflection-tanh family and are not a guarantee for arbitrary nonexpansive mappings.
As expected, all methods converge more slowly here than in the contractive setting, reflecting the weaker assumptions on the underlying mapping.

\vspace{1ex}
\noindent\textbf{Acknowledgements.}
This work was partially supported by the National Science Foundation (NSF) through the RTG grant DMS-2134107 and by the Office of Naval Research (ONR) under grant No.~N00014-23-1-2588. 
A substantial portion of this work was completed while the first two authors, Quoc Tran-Dinh and Pham Ngoc Anh, were visiting the Vietnam Institute for Advanced Study in Mathematics (VIASM), Hanoi, Vietnam, in June 2026.

\vspace{1ex}
\noindent{\textbf{Conflicts of interest/competing interests.}}
The authors declare that they have no conflicts of interest or competing interests related to this work.

\vspace{1ex}
\noindent{\textbf{Data availability.}}
This paper uses only synthetic data.
The procedure for generating synthetic data is clearly described in the paper.

\vspace{1ex}
\noindent\textbf{Declaration of AI-Assisted Tools.}
During the preparation of this manuscript, we used large language models, including ChatGPT and Gemini, to assist in verifying elementary mathematical derivations, algebraic manipulations, and properties of mathematical expressions. We also used Codex to assist with implementing the proposed algorithms, generating synthetic data, and preparing numerical experiments.
All mathematical proofs, algorithmic implementations, numerical results, and conclusions presented in this paper were carefully reviewed and independently verified by the authors. The authors take full responsibility for the correctness and integrity of the content of this manuscript.

\bibliographystyle{plain}

\begin{thebibliography}{10}
\itemsep=-0.0em
\bibitem{alakoya2021modified}
T.~Alakoya, L.~Jolaoso, and O.~T. Mewomo.
\newblock Modified inertial subgradient extragradient method with self adaptive
  stepsize for solving monotone variational inequality and fixed point
  problems.
\newblock {\em Optimization}, 70(3):545--574, 2021.

\bibitem{attouch2020convergence}
H.~Attouch and A.~Cabot.
\newblock Convergence of a relaxed inertial proximal algorithm for maximally
  monotone operators.
\newblock {\em Math. Program.}, 184(1):243--287, 2020.

\bibitem{Banach1922}
S.~Banach.
\newblock Sur les opérations dans les ensembles abstraits et leur application
  aux équations intégrales.
\newblock {\em Fundamenta Mathematicae}, 3:133--181, 1922.

\bibitem{Barzilai1988}
J.~Barzilai and J.~M. Borwein.
\newblock {Two-Point Step Size Gradient Methods}.
\newblock {\em IMA J Numer Anal}, 8(1):141--148, January 1988.

\bibitem{Bauschke2011}
H.~H. Bauschke and P.~Combettes.
\newblock {\em {C}onvex {A}nalysis and {M}onotone {O}perator {T}heory in
  {H}ilbert {S}paces}.
\newblock Springer-Verlag, 2nd edition, 2017.

\bibitem{bot2022bfast}
R.~I. Bot and D.~K. Nguyen.
\newblock Fast {K}rasnosel\'skii-{M}ann algorithm with a convergence rate of
  the fixed point iteration of $o(1/k)$.
\newblock {\em SIAM J. Numer. Anal.}, 61(6):2813--2843, 2023.

\bibitem{ioan2023relaxed}
R.~I. Bot, M.~Sedlmayer, and P.~T. Vuong.
\newblock A relaxed inertial forward-backward-forward algorithm for solving
  monotone inclusions with application to {GANs}.
\newblock {\em J. Mach. Learn. Res. (JMLR)}, 24:1--37, 2023.

\bibitem{carmon2022making}
Y.~Carmon and O.~Hinder.
\newblock Making {SGD} parameter-free.
\newblock In {\em Conference on learning theory}, pages 2360--2389. PMLR, 2022.

\bibitem{colao2015krasnoselskii}
V.~Colao and G.~Marino.
\newblock {K}rasnoselskii-{M}ann method for non-self mappings.
\newblock {\em Fixed Point Theory Appl.}, 2015(1):39, 2015.

\bibitem{daskalakis2018training}
C.~Daskalakis, A.~Ilyas, V.~Syrgkanis, and H.~Zeng.
\newblock Training {GANs} with {O}ptimism.
\newblock In {\em International Conference on Learning Representations (ICLR
  2018)}, pages 1--9, 2018.

\bibitem{defazio2022parameter}
A.~Defazio and K.~Mishchenko.
\newblock Parameter free dual averaging: {O}ptimizing {L}ipschitz functions in
  a single pass.
\newblock In {\em OPT 2022: Optimization for Machine Learning (NeurIPS 2022
  Workshop)}, 2022.

\bibitem{defazio2023learning}
A.~Defazio and K.~Mishchenko.
\newblock Learning-rate-free learning by d-adaptation.
\newblock In {\em International conference on machine learning}, pages
  7449--7479. PMLR, 2023.

\bibitem{diakonikolas2020halpern}
J.~Diakonikolas.
\newblock {H}alpern iteration for near-optimal and parameter-free monotone
  inclusion and strong solutions to variational inequalities.
\newblock In {\em Conference on Learning Theory}, pages 1428--1451. PMLR, 2020.

\bibitem{Duchi2011}
J.~Duchi, E.~Hazan, and Y.~Singer.
\newblock Adaptive subgradient methods for online learning and stochastic
  optimization.
\newblock {\em J. Mach. Learn. Res.}, 12:2121--2159, 2011.

\bibitem{Facchinei2003}
F.~Facchinei and J.-S. Pang.
\newblock {\em Finite-dimensional variational inequalities and complementarity
  problems}, volume 1-2.
\newblock Springer-Verlag, 2003.

\bibitem{fercoq2019adaptive}
O.~Fercoq and Z.~Qu.
\newblock Adaptive restart of accelerated gradient methods under local
  quadratic growth condition.
\newblock {\em IMA J. Numer. Anal.}, 39(4):2069--2095, 2019.

\bibitem{halpern1967fixed}
B.~Halpern.
\newblock Fixed points of nonexpanding maps.
\newblock {\em Bull. Am. Math. Soc.}, 73(6):957--961, 1967.

\bibitem{he2021optimal}
S.~He, Q.-L. Dong, H.~Tian, and X.-H. Li.
\newblock On the optimal relaxation parameters of {K}rasnosel'ski--{M}ann
  iteration.
\newblock {\em Optimization}, 70(9):1959--1986, 2021.

\bibitem{he2018totally}
S.~He, T.~Wu, A.~Gibali, and Q.-L. Dong.
\newblock Totally relaxed, self-adaptive algorithm for solving variational
  inequalities over the intersection of sub-level sets.
\newblock {\em Optimization}, 67(9):1487--1504, 2018.

\bibitem{he2024convergence}
S.~He, H.-K. Xu, Q-L. Dong, and N.~Mei.
\newblock Convergence analysis of the {H}alpern iteration with adaptive
  anchoring parameters.
\newblock {\em Math. Comput.}, 93(345):327--345, 2024.

\bibitem{ito2023parameter}
M.~Ito, Z.~Lu, and C.~He.
\newblock A parameter-free conditional gradient method for composite
  minimization under {H}{\"o}lder condition.
\newblock {\em J. Mach. Learn. Res.}, 24(166):1--34, 2023.

\bibitem{KingmaB14}
D.~P. Kingma and J.~Ba.
\newblock Adam: A method for stochastic optimization.
\newblock {\em CoRR}, abs/1412.6980, 2014.

\bibitem{Konnov2001}
I.V. Konnov.
\newblock {\em Combined relaxation methods for variational inequalities.}
\newblock Springer-Verlag, 2001.

\bibitem{Krasnoselskii1955}
M.~A. Krasnosel'ski\v{i}.
\newblock Two remarks on the method of successive approximations.
\newblock {\em Uspekhi Matematicheskikh Nauk}, 10(1):123--127, 1955.
\newblock In Russian.

\bibitem{lan2026optimal}
G.~Lan, Y.~Ouyang, and Z.~Zhang.
\newblock Optimal and parameter-free gradient minimization methods for convex
  and nonconvex optimization.
\newblock {\em Math. Program.}, pages 1--40, 2026.

\bibitem{li2025simple}
T.~Li and G.~Lan.
\newblock A simple uniformly optimal method without line search for convex
  optimization.
\newblock {\em Math. Program.}, pages 1--38, 2025.

\bibitem{lieder2021convergence}
F.~Lieder.
\newblock On the convergence rate of the {H}alpern-iteration.
\newblock {\em Optim. Letters}, 15(2):405--418, 2021.

\bibitem{lv2026preconditioned}
F.~Lv and Q.-L. Dong.
\newblock Preconditioned {H}alpern iteration with adaptive anchoring parameters
  and an acceleration to {C}hambolle--{P}ock algorithm.
\newblock {\em Numer. Algorithms}, pages 1--30, 2026.

\bibitem{malitsky2019golden}
Y.~Malitsky.
\newblock Golden ratio algorithms for variational inequalities.
\newblock {\em Math. Program.}, 184(1--2):383--410, 2020.

\bibitem{malitsky2020adaptive}
Y.~Malitsky and K.~Mishchenko.
\newblock Adaptive gradient descent without descent.
\newblock In {\em International Conference on Machine Learning}, pages
  6702--6712. PMLR, 2020.

\bibitem{malitsky2016first}
Y.~Malitsky and T.~Pock.
\newblock A first-order primal-dual algorithm with linesearch.
\newblock {\em SIAM J. Optim.}, 28(1):411--432, 2018.

\bibitem{malitsky2020forward}
Y.~Malitsky and M.~K. Tam.
\newblock A forward-backward splitting method for monotone inclusions without
  cocoercivity.
\newblock {\em SIAM J. Optim.}, 30(2):1451--1472, 2020.

\bibitem{mann1953mean}
W.~R. Mann.
\newblock Mean value methods in iteration.
\newblock {\em Proc. Amer. Math. Soc.}, 4(3):506--510, 1953.

\bibitem{ogwo2025inertial}
G.~N. Ogwo, B.~Zinsou, H.~A. Abass, and O.~K. Oyewole.
\newblock Inertial {H}alpern-type {T}seng’s method for approximating a
  solution to monotone inclusion problems with fixed point constraint.
\newblock {\em Rendiconti del Circolo Matematico di Palermo Series 2},
  74(1):61, 2025.

\bibitem{orabona2016coin}
F.~Orabona and D.~P{\'a}l.
\newblock Coin betting and parameter-free online learning.
\newblock {\em Advances in Neural Information Processing Systems}, 29, 2016.

\bibitem{oyewole2022totally}
O.~K. Oyewole and S.~Reich.
\newblock A totally relaxed self-adaptive algorithm for solving a variational
  inequality and fixed point problems in banach spaces.
\newblock {\em Appl Set-Valued Anal Optim}, 4:349--366, 2022.

\bibitem{park2022exact}
J.~Park and E.~K. Ryu.
\newblock Exact optimal accelerated complexity for fixed-point iterations.
\newblock In {\em International Conference on Machine Learning}, pages
  17420--17457. PMLR, 2022.

\bibitem{ryu2016primer}
E.~K. Ryu and S.~Boyd.
\newblock Primer on monotone operator methods.
\newblock {\em Appl. Comput. Math}, 15(1):3--43, 2016.

\bibitem{sabach2017first}
S.~Sabach and S.~Shtern.
\newblock A first order method for solving convex bilevel optimization
  problems.
\newblock {\em SIAM J. Optim.}, 27(2):640--660, 2017.

\bibitem{shen2026parameter}
L.~Shen and F.~K{\i}l{\i}n{\c{c}}-Karzan.
\newblock Parameter-free non-ergodic extragradient algorithms for solving
  monotone variational inequalities.
\newblock {\em arXiv preprint arXiv:2604.07662}, 2026.

\bibitem{tan2022self}
B.~Tan and X.~Qin.
\newblock Self adaptive viscosity-type inertial extragradient algorithms for
  solving variational inequalities with applications.
\newblock {\em Math. Model. Anal.}, 27(1):41--58, 2022.

\bibitem{tran2022connection}
Q.~Tran-Dinh.
\newblock From {H}alpern's fixed-point iterations to {N}esterov's accelerated
  interpretations for root-finding problems.
\newblock {\em Comput. Optim. Appl.}, 87(1):181--218, 2024.

\end{thebibliography}

\end{document}